\documentclass[11pt,letterpaper]{amsart}
\usepackage{amsmath, amssymb, latexsym, amsthm, amsfonts, mathrsfs, amscd}
\usepackage{enumerate}[1.)]
\renewcommand{\eqref}[1]{(\ref{#1})}   
\usepackage{showtags} 
\usepackage{enumitem}
\usepackage{latexsym}
\usepackage{euscript}
\usepackage[margin=1.25in]{geometry}
\usepackage{epsfig}
\usepackage{hyperref}
\usepackage{tikz}
\usepackage{bbm}
\usepackage{graphics}
\usepackage{array}
\usepackage[utf8]{inputenc}
\usepackage{xcolor}
\usepackage{dsfont}
\newtheorem{theorem}{Theorem}[section]
\newtheorem{lemma}[theorem]{Lemma}
\newtheorem{proposition}{Proposition}[section]
\theoremstyle{definition}
\newtheorem{definition}[theorem]{Definition}

\newtheorem{remark}{Remark}

\renewcommand{\l}{\left}
\renewcommand{\r}{\right}
\newtheorem{corollary}[theorem]{Corollary}

\newcommand{\modq}{(\textnormal{mod }q)}

\newcommand{\diff}{\textnormal{d}}
\newcommand{\Q}{{\mathbb Q}}
\newcommand{\R}{{\mathbb R}}
\newcommand{\Z}{{\mathbb Z}}
\newcommand{\C}{{\mathbb C}}
\newcommand{\N}{{\mathbb N}}

\newcommand{\be}{\begin{equation}}
\newcommand{\ee}{\end{equation}}
\newcommand{\ba}{\begin{equation}\begin{aligned}}
\newcommand{\ea}{\end{aligned}\end{equation}}
\numberwithin{equation}{section}

\usepackage{setspace}
\usepackage{appendix}

\newcommand{\ve}{\varepsilon}

\usepackage[most]{tcolorbox}

\begin{document}
	\title{On quadric surfaces weighted by arbitrary coefficients } 
	\author{Rachita Guria \and Prahlad Sharma \and Saurabh Singh}
    \address{Indian Institute of Technology Bombay, Powai, Mumbai, Maharashtra--400076}
	\email{guriarachita2011@gmail.com}
    
     \address{Indian Institute of Technology Kanpur, Kalyanpur, Kanpur--208016}
     \email{prahladks@iitk.ac.in}

       \address{Indian Institute of Technology Kanpur, Kalyanpur, Kanpur--208016}
     \email{saurabs@iitk.ac.in}

	\maketitle
	
\begin{abstract}
For a non-singular quadratic form $F$ in four variables, we study integer solutions to $F=0$ with an arbitrary weight attached to one of the variables. This allows us to investigate genuinely sparse subsets of the corresponding variety. Our main input is to obtain cancellation in sums of character sums over the moduli arising in the application of the DFI's delta method to the quadratic hypersurface. This cancellation relies crucially on the equidistribution of quadratic roots and the Weil bound for Kloosterman sums.
\end{abstract}

\section{Introduction}\label{introandresult}
The study of rational points on varieties has long been a fundamental theme in number theory. Early examples include Lagrange’s four-square theorem \cite{lagrange} and Waring’s problem \cite{waring}, concerning the representation of integers as sums of powers. The latter became a major motivation for subsequent counting methods. A major advance in the quantitative study of these integral solutions came with the circle method, originating in Hardy and Ramanujan’s work \cite{hardy-ramanujan} on partitions  and further developed by Hardy and Littlewood in their treatment of Waring’s problem \cite{hardy-littlewood}. 

However, the implicit nature of harmonics involved makes it difficult to adapt the circle method as the number of variables decreases. For quadratic forms, the case considered in this paper, at least five variables are required to obtain an asymptotic formula for the number of representations of an integer by the form. The barrier of five variables was overcome and reduced to four by Kloosterman \cite{Klooster} using his refinement of the circle method. Now known as the Kloosterman refinement, his approach exploits cancellation in sums over the numerators of the approximating Farey fractions, giving rise to what are now called Kloosterman sums. This still left the important problem of obtaining an asymptotic formula for the number of integer solutions to the homogeneous equation $F=0$, where $F$ is a quadratic form in four variables. Heath-Brown \cite{HB} later resolved this case using a modern form of the circle method known as the delta method. The delta-symbol approach avoids the major/minor arc dichotomy and is more explicit with the harmonics involved. Heath-Brown exploited this flexibility to obtain cancellation in both the numerators and denominators of certain approximating fractions arising in the method, which can be viewed as a ``double Kloosterman refinement''.

The works discussed above and related results mainly deal with unrestricted integer solutions, while relatively few address solutions in sparse subsets of the integers. One direction that has been pursued is the study of solutions in primes or product of primes--a subject central to the Waring–Goldbach problem, which have also been investigated in many different settings; see \cite{hua1947additive, brudern1994lagrange, heath2003lagranges, tolev2002lagrange, tsang2017lagrange, greaves1976representation, plaksin1982asymptotic, plaksin1985asymptotic, WesselZurVerth, julia, GG, guria2026primesolutionsbinaryadditive}. A principal difficulty for more general sparse sets is that the associated exponential sums require new estimates that are not readily available. Consequently, one cannot expect to study genuinely sparse solutions using the traditional approach. 

In this paper, we address this problem for quadratic forms in four variables. More precisely, we obtain asymptotic formulas for solutions to $F=0$, with an essentially arbitrary weight attached to the first variable and obtain an error term in which the arithmetic dependence is isolated into simple norms of the weight. Applications include solutions with one variable restricted to perfect powers, short intervals of length essentially as small as $1$, or primes, as well as counts of base points whose associated conic fibres are soluble over $\mathbb{Q}$. The leading constants in the main term reveal interesting local-global principles in each of these examples. Furthermore, for general multiplicative weights, we give a clean formulation of the main term in terms of the associated $L$-function. Examples include the indicator function of square-free numbers, divisor functions, the Möbius function $\mu(n)$, and coefficients of general $GL(d)$ forms.

Our main tool is the Duke-Friedlander-Iwaniec delta symbol expansion. It turns out delta symbol packs enough harmonics to provide sufficient cancellation without exploiting the arithmetic properties of one of the variables. More precisely, alongside the three unrestricted variables, we exploit the sum over moduli in the delta-symbol expansion to obtain cancellation in sums of certain character sums, in place of the cancellation in the fourth variable used in Heath-Brown’s approach. As we shall see, the cancellation in these sums comes from the equidistribution of quadratic roots and the Weil bound for Kloosterman sums. Using the ideas in this paper, one can study quadric surfaces with arbitrary weights attached to more variables as the total number of variables increases. The character sum analysis will be different in each case and will be explored in subsequent works.

\vspace{4mm}

Let $ F(x_1, x_2, x_3, x_4)$  be a non-singular integer quadratic form in 4 variables with the associated symmetric matrix $M\in M_4(\mathbb{Z})$. The form is then given by
\begin{align}\notag
    F(\textbf{x})=\textbf{x}^{T}M\textbf{x},\,\,\,\textbf{x}=(x_1,x_2,x_3,x_4).
\end{align}In this paper, we focus on the case where $\det M$ is not a square. We define the quadratic character $\chi_M$ by
\begin{align*}
    \chi_M(n)=\left(\frac{\det M}{n}\right),
\end{align*}
which will play an important role in our analysis. We also fix a compactly supported smooth weight function $w:(0,\infty)^4\to\mathbb{C}$ for our counting problem, localising the solutions to its support.

Since the problem involves counting solutions keeping the first variable fixed, the following product of restricted local densities of solutions naturally arise in our asymptotic:
\begin{equation}\label{eq:singular-series-product}
\mathfrak{S}_{F}(n)
=
\prod_{p}\sigma_{p}(n)  \ \ \ \ \textrm{with} 
\end{equation}
\begin{equation}\label{eq:local-density}
\sigma_{p}(n):=
\lim_{\nu\to\infty}
p^{-2\nu}
\#\left\{
\mathbf{b}\bmod p^{\nu}:
F(n,\mathbf{b})\equiv 0\pmod{p^{\nu}}
\right\}.
\end{equation}and
\begin{equation*}
   \sigma_{\infty}(a)=\lim_{\varepsilon\to 0} \frac{1}{2\varepsilon}\int_{|F(a,\textbf{t})|\leq \varepsilon}w(a,\textbf{t})\,d\textbf{t}. 
\end{equation*}The product $\prod_p \sigma_p(n)$ is interpreted as $L(1,\chi_M)\prod_p(1-\chi_M(p)/p)\sigma_p(n)$, where the Euler product on the right converges absolutely (see \eqref{localfactoreval}). Our main result is 

\begin{theorem}\label{mainresult}
    Let $F$ be a non-singular quadratic form in four variables whose determinant is not a square. Suppose further that the submatrix obtained by deleting the first row and first column of the matrix defining $F$ is also non-singular. Then, for any complex sequence $\{A(n)\}_{n\geq 1}$ we have
    \begin{equation*}
    \sum_{\substack{\mathbf{x}\in \mathbb{Z}^4 \\ F(\mathbf{x}) = 0 }} A(x_1) w\l(\frac{\mathbf{x}}{X} \r)=X \sum_{n\geq 1} A(n)\sigma_{\infty}(n/X) \mathfrak{S}_F(n) + \mathcal{R}_A(X), 
    \end{equation*}where
    \begin{equation*}
        \mathfrak{S}_{F}(n)=\prod_{p}\sigma_{p}(n),
    \end{equation*} and the error $\mathcal{R}_A(X)$ is given by 
\begin{equation*}
   \mathcal{R}_A(X)\ll  X^{7/8}\|A\|+X\|A\| \l(R^{-1/130} + R^{1/4}\|A\|_2\|A\|^{-1} + R\cdot E_A(X,R)\|A\|^{-1}\r)
\end{equation*}for any $R>1$ and $E_A(X,R)$ is the correlation
\begin{equation}\label{correlationsum}
    E_A(X,R)=\sup_{c\ll R,\,\,t\ll X/c}\left|\sum_{n\leq t}A(cn)\chi_M(n)\right|.
\end{equation}
\end{theorem}
 Here, and throughout $\|A\|$ and $\|A\|_2$ denotes the restricted norms
 \begin{equation*}
     \sum_{X\leq n\leq 2X}|A(n)|\,\,\,\text{and}\,\,\,\sum_{X\leq n\leq 2X}|A(n)|^2
 \end{equation*}respectively. The main term in Theorem \ref{mainresult} is roughly of size $\|A\|X$. Hence, all but the last term in $\mathcal{R}_A(X)$ are dominated by the main term whenever $1\ll R\ll \|A\|^2$. Note that the conductor of $\chi_M$ is $\det M$, which is fixed throughout. So if $A(n)$ does not correlate with the quadratic character $\chi_M$, then $E_A(X,R)\|A\|^{-1}$ is typically small compared with $1$ and becomes smaller as $R$ is reduced.  Thus, in typical applications, one can choose $R$ sufficiently small relative to $X$ so that $RE_A(X,R)\|A\|^{-1}=o(1)$ as $X\to\infty$. Together, these observations imply that the above asymptotic is non-trivial for small enough $R\ll X$. If the twisted sum $\sum_{n}A(n)\chi_M(n)$ has power savings, this translates to a power-saving error in the asymptotic.

\vspace{5pt}

The function $\sigma_{\infty}(x)$ is smooth (see Proposition \ref{presingularint}), and the restricted singular series admits a useful factorisation in terms of the prime-power valuations of \(n\); see equation \eqref{factorisation}. This allows us to express them as a product of local factors determined by the prime factorisation of $n$. Consequently, the main term in Theorem \ref{mainresult} can be readily evaluated for most sequences $A(n)$. We give some concrete examples below to illustrate this.
\subsection{Applications}

\subsubsection{\textbf{Multiplicative sequences}}

Suppose $A(n)$ is multiplicative and satisfy the bound $A(p^t)\ll_{\varepsilon} p^{ t(\theta+\varepsilon)}$ for some $\theta>0$. Assume that the associated Dirichlet series
\begin{equation*}
    L(A,s)=\sum_{n\geq 1}A(n)n^{-s},\,\Re(s)>1+\theta
\end{equation*}extends meromorphically to the plane $\Re(s)>\theta>0$, with at most a simple pole at $s=\sigma_0 \in (\theta, 1+\theta]$. Assume also that $L(A,s)$ has at most polynomial growth in $|\Im (s|)$ on every closed vertical strip contained in this half-plane, away from its possible pole, and that each Euler factor $A_p(s)=\sum_{k\ge0}A(p^k)p^{-ks}$ is non-vanishing for $\Re(s)>\theta$.

\begin{theorem}\label{mainresult2}
    For $F$ as in Theorem \ref{mainresult} and $A(n)$ as above, 
   \begin{equation*}
       \sum_{\substack{\mathbf{x}\in \mathbb{Z}^4 \\ F(\mathbf{x}) = 0 }} A(x_1) w\l(\frac{\mathbf{x}}{X} \r)=X^{1+\sigma_0}\mathfrak{C}+O_{\varepsilon}(X^{1+\theta+\varepsilon})+\mathcal{R}_A(X),
   \end{equation*}where
   \begin{equation*}
      \mathfrak{C}=L(1,\chi)\sigma_{\infty}(F,A,w)\sigma^*(F,A)\,\operatorname*{Res}_{s=\sigma_0}L(A,s),
   \end{equation*}where $\sigma^*$ is the weighted singular series
   \begin{equation*}
       \sigma^*(F,A)=\prod_{p}\frac{(1-\chi(p)/p)\l(\sum_{k\geq   0} A(p^k) \sigma_p(p^{k})p^{-k\sigma_0}\r)}{A_p(\sigma_0)},
   \end{equation*}and $\sigma_{\infty}$ is the weighted singular integral
   \begin{equation*}
       \sigma_{\infty}(F,A,w)=\lim_{\varepsilon\to 0}\frac{1}{2\varepsilon}\int_{|F(\textbf{x})|\leq \varepsilon}x_1^{\sigma_0-1}w(\textbf{x})\,d\textbf{x}.
   \end{equation*}
\end{theorem}Theorem \ref{mainresult2} yields explicit asymptotic formulas for a wide range of interesting arithmetic functions, including the M\"{o}bius function $\mu(n)$, divisor functions, coefficients of cuspidal $GL(d)$ automorphic forms, and the constant function $\mathbbm{1}$. The automorphic application generalizes the work of V. Kumaraswamy \cite{VK} from degree two to arbitrary degree, while the constant function recovers Heath-Brown's result, albeit with a weaker error term due to the generality of our argument. We do not spell out the corresponding asymptotic formulas here. Instead, we highlight the following application to the indicator function of $k$-th powers, which counts solutions on the surface whose first coordinate is a $k$-th power, for $k\in\mathbb{N}$.

For $A=\mathbbm{1}_{\{m^k: \ m\in\mathbb{N}\}}$, note that $ L(A,s)=\zeta(ks),$ which has a simple pole at $s=1/k$ with residue $1/k$. Applying Theorem \ref{mainresult2} with $\theta=0$, and using Proposition \eqref{kthpowerseries} and Lemma \ref{kth-power-correlation} to evaluate the relevant singular series and the error term $\mathcal{R}_A(X)$, respectively, we obtain :
\begin{corollary}\label{kpower}
    For $k\in\mathbb{N}$ odd and $F$ as in Theorem \ref{mainresult}, we have
    \begin{equation*}
    \begin{aligned}
\sum_{\substack{\mathbf{x}\in \mathbb{Z}^4 \\ x_1=m^k,\,m\in\mathbb{N}\\  F(\mathbf{x}) = 0 }}  w\l(\frac{\mathbf{x}}{X} \r)&=\frac{1}{k}X^{1+1/k}L(1,\chi_M)\sigma^*(F,k)\sigma_{\infty}(F,w,k)\\
&+ O_F\left(X^{1+1/k-1/8}+X^{1+1/k-1/(131k)}\right).
\end{aligned}
    \end{equation*}where $\sigma^*(F,k)=\prod_{p}(1-\chi(p)/p)\sigma_{p,k}$ is the absolutely converging  product of associated local densities
    \begin{equation*}
         \sigma_{p,k}=\lim_{\nu\to\infty}p^{-3\nu}\#\{\textbf{b}\in (\mathbb{Z}/p^\nu\mathbb{Z})^4 : F(b_1^k,b_2,b_3,b_4)\equiv 0\pmod {p^{\nu}}\}
    \end{equation*}and 
    \begin{equation*}
        \sigma_{\infty}(F,w,k)=\lim_{\varepsilon\to 0}\frac{1}{2\varepsilon}\int_{|F(\textbf{x})|\leq \varepsilon}x_1^{1/k-1}w(\textbf{x})\,d\textbf{x}.
    \end{equation*}
\end{corollary}Corollary \ref{kpower} establishes a quantitative Hasse principle for $G(\textbf{x})=F(x_1^k,x_2,x_3,x_4)$, which is a  polynomial of degree $2k$. For non-singular $F$, one can show elementarily that $\sigma_{p,k}>0$ if and only if $\sigma_{p,1}>0$. Hence, the above result shows that $G(\textbf{x})=0$ has the expected number of solutions as soon as the original form $F$ has a $p$-adic solution for every prime $p$ and a real solution in the support of $w$. See Remark \ref{remarkevenk} below for the case of even powers. 

\subsubsection{\textbf{Soluble conic fibres}} Fix a imaginary quadratic field $K$ and let $Q$ be the norm form of its full ring of integers. We give a count on the integral points $\mathbf{x}$ on the base $F(\mathbf{x})=0$ for which the corresponding conic 
\begin{equation*}
    C_{\mathbf{x}}: Q(u,v)=x_1w^2,
\end{equation*}is soluble over $\mathbb{Q}$. For simplicity, we assume that $K$ has class number one so that the coefficient
\begin{equation*}
A(x_1)=\mathbbm1_{w^2x_1=Q(u,v)},
\end{equation*}is multiplicative in $x_1$. For general class number, this coefficients is a finite combination of multiplicative functions, and the analysis can be carried out similarly. Using  \ref{mainresult} with the above coefficients $A(x_1)$ coupled with evaluations from Lemma \ref{conicmt} and the error estimation from \ref{conicerror} we get :
\begin{theorem}\label{conicfibres}
    Let be $F$ as in Theorem \ref{mainresult} and $K$ be a imaginary quadratic field with class number one and fundamental discriminant $D$. Suppose further that $\chi\chi_D$ is non-principle. Then,
\begin{equation*}
    \sum_{\substack{\textbf{x}\in\mathbb{Z}^4\\F(\textbf{x})=0\\C_{\textbf{x}}(\mathbb{Q})\neq \varnothing}}w\l(\frac{\mathbf{x}}{X}\r)=\frac{X^2}{\sqrt{\log X}}L(1,\chi)C_K\,\,\sigma^*(F,K)\sigma_{\infty}(F,w)+O\l(\frac{X}{(\log X)^{3/2}}\r),
\end{equation*}where $C_K=L(1,\chi_D)^{1/2}\prod_{\chi_D(p)=-1}(1-p^{-2})^{1/2}
\prod_{p|D}(1-p^{-1})^{-1/2}$ and $$ \sigma^*(F,K)=\prod_{p}(1-\chi(p)/p)\sigma^*_p,$$ where
\begin{equation*}
\sigma^*_p=\begin{cases}
    (1+p^{-1})\sigma_{p,2}, & \chi_D(p)=-1,\\
    \sigma_{p,1}, & \chi_D(p)\neq -1.
\end{cases}
\end{equation*}Here $\sigma_{p,2}$ and $\sigma_{p,1}$ are as in Corollary \ref{kpower}. 
\end{theorem}Note that when $\chi\chi_D$ is the principle character, this is precisely when the coefficients $A$ correlates with $\chi_M$. The result can be extended for this case with some extra care; see Remark \ref{remarkevenk}.

The above result is an affine counterpart of the soluble-fibre counting problem studied by Loughran and Smeets \cite{LoughranSmeets}. In the projective setting, Browning and Loughran \cite[Theorem 1.10]{BrowningLoughran} obtained upper bounds of the conjectured order for fibrations over smooth quadrics of dimension at least three, corresponding to quadratic forms in at least five variables. Their quantitative theorem does not cover quadric surfaces. Asymptotic formulas for certain conic bundles over higher-dimensional hypersurfaces were subsequently established by Sofos and Visse-Martindale \cite{SofosVisseMartindale}. Our application provides an explicit asymptotic in the affine setting associated with a quadratic form in four variables.

\subsubsection{\textbf{Solutions in short intervals}}
As remarked earlier the local coefficient $\sigma_{\infty}(x/X)$ is smooth as a function of $x$ and $1$-inert so that one has
\begin{equation*}
    \int_{\mathbb{R}}\l|\frac{\partial \sigma_{\infty}(t/X)}{\partial t}\r|\,dt\ll 1.
\end{equation*}Hence applying Theorem \ref{mainresult} for indicator function of short interval $\mathbbm{1}_{[X,\ X+H]}$ and using Corollary \ref{shortintervalsingularseries} for the evaluation of the  $\sum_{X\leq n\leq X+H}\mathfrak{S}_F(n)\sigma_{\infty}(n/X)$, we obtain :
\begin{theorem}\label{shortinterval}
    Let $F$ be as in Theorem \ref{mainresult}. For any $1\leq H\leq X$, we have
    \begin{equation*}
        \sum_{\substack{\mathbf{x}\in \mathbb{Z}^4 \\ x_1\in [X,X+H]\\  F(\mathbf{x}) = 0 }}  w\l(\frac{\mathbf{x}}{X} \r)=XHL(1,\chi)\sigma^*(F)\sigma_{\infty}(F,H)+  O \left(  X^{7/8} H +XH^{130/131}\right), 
    \end{equation*}where $\sigma^*(F)=\prod_{p}(1-\chi(p)/p)\sigma_p$, and 
\begin{equation*}
         \sigma_{p}=\lim_{\nu\to\infty}p^{-3\nu}\#\{\textbf{b}\in (\mathbb{Z}/p^\nu\mathbb{Z})^4 : F(b_1,b_2,b_3,b_4)\equiv 0\pmod {p^{\nu}}\},
    \end{equation*}and
    \begin{equation*}
        \sigma_{\infty}(F,H)=\lim_{\varepsilon\to 0}\frac{1}{2\varepsilon}\int_{|F(\textbf{x})|\leq \varepsilon}\mathbbm{1}_{[1,1+H/X]}(x_1)w(\textbf{x})\,d\textbf{x}.
    \end{equation*}
\end{theorem}Note that Theorem \ref{shortinterval} is non-trivial for $H$ as small as $X^{\varepsilon}$ for any fixed $\varepsilon>0$. It establishes a quantitative Hasse principle in the prescribed region $X\leq x_1\leq X+H$, for any $H\gg X^{\varepsilon}$ with fixed $\varepsilon>0$: the number of solutions to $F(\mathbf{x})=0$ in this region has the expected asymptotic with a positive main term if and only if the variety has a nontrivial $p$-adic solution for every prime $p$ and a real solution whose first coordinate lies in $[1,1+H/X]$.

\subsubsection{\textbf{Prime solutions}} We next show that the integer solutions to $F=0$ with prime first coordinate occur with the same order of relative density as primes among the natural numbers. Indeed, taking $A$ to be the prime indicator function in Theorem \ref{mainresult} and using Corollary \ref{primesingularseries} with $g(x)=\sigma_{\infty}(x/X)$ we obtain
\begin{theorem}\label{primeresult}
     Let $F$ be as in Theorem \ref{mainresult}. There exists a positive constant $ c_F$ such that 
\begin{equation*}
\sum_{\substack{\mathbf{x}\in \mathbb{Z}^4 \\ x_1\,\text{prime}\\  F(\mathbf{x}) = 0 }}  w\l(\frac{\mathbf{x}}{X} \r) =\frac{X^2}{\log X}L(1,\chi)\sigma^*(F)\sigma_{\infty}(F,w)+O\l(X^2e^{-c_F\sqrt{\log X}}\r),
     \end{equation*}where $ \sigma^*(F)=\prod_p(1-\chi(p)/p)\sigma_p$, with 
 \begin{equation*}
\sigma_p=\lim_{\nu\to\infty} p^{-2\nu} \#\left\{\mathbf{b}\bmod p^{\nu}:
F(1,\mathbf{b})\equiv 0\pmod{p^{\nu}} \right\}, 
\end{equation*}and
\begin{equation*}
\sigma_{\infty}(F,w)=\lim_{\varepsilon\to 0}\frac{1}{2\varepsilon}\int_{|F(\textbf{x})|\leq \varepsilon}w(\textbf{x})\,d\textbf{x}.
    \end{equation*}
\end{theorem}The constants $\sigma^*$ and $\sigma_{\infty}$ are positive provided that $F(1,\mathbf{b})=0$ has a solution in $\mathbb{Z}_p^3$ for every prime $p$ and that $F=0$ has a nontrivial real solution at which $w$ is positive. Theorem \ref{primeresult} establishes the following local–global principle: the number of integer solutions to $F=0$ with prime first coordinate has the expected asymptotic with a positive main term if and only if $F(1,\mathbf{b})=0$ has a solution in $\mathbb{Z}_p^3$ for every prime $p$ and a non-trivial real solution. 

\vspace{3mm}

\textbf{Remarks.}
\begin{enumerate}
  \item\label{remarkevenk} When $\det M$ is a square or when $A$ correlates with the quadratic character $\chi_M$ (for example, when $A$ is supported on squares), another source of main term will be from $\sum_{n}A(n)\chi_M(n)$ appearing in \eqref{correlationsum}. This can be extracted out and our theorem can be extended to cover these cases which we decided to omit to keep the exposition simple.
  \item The method of this paper can also be easily adapted to obtain asymptotics for the closely related binary sums
\begin{equation}\notag
\sum_{n}\sum_{m}A(n)\lambda _f(P(n,m)),
\end{equation}
where $\lambda_f(n)$ are $GL(2)$ coefficients, including the divisor function, and $P(n,m)$ is a binary quadratic form. These sums have studied in special cases when $\lambda_f(n)=d(n)$ and the form $P(n,m)$ have a specific shape (see  \cite{quad1, quad2, quad3, TD3, Daniel99, quad6, quad7, ChananaSingh}). The adaptation of our method to this setting would extend and generalise the all these previous results in a uniform manner.
  \item A specific degeneracy arises in our argument when the coefficient of $c_1^2$ in the dual form $F^*(\mathbf{c})$ vanishes, which occurs precisely when the submatrix $N$ obtained by deleting the first row and first column of $M$ has determinant zero. In fact, the restriction $\det N\neq 0$ in our theorem is quite natural since in this case the form $F$ essentially splits as a sum of two binary quadratic forms and can be treated in a more elementary fashion. To see this, let
\begin{align}\notag
M=
\begin{pmatrix}
m&\textbf{r}^T\\
\textbf{r}&N
\end{pmatrix} 
\end{align} be the matrix defining $F$ where $\det N=0$. Choose $\mathbf{v}\in\mathbb{Z}^3\backslash\mathbf{0} $ with $N\mathbf{v}=\textbf{0}$ and let $N_0\in GL_3(\mathbb{Z})$ be such that the last column of $N_0$ is $\mathbf{v}$. Denote 
\begin{align}\notag
M_0=
\begin{pmatrix}
1&\textbf{0}\\
\textbf{0}&N_0
\end{pmatrix} 
\end{align}Then 
\[M_0^TMM_0=
\begin{pmatrix}
m & \ast & \ast & \mathbf{r}^T\mathbf{v} \\
\ast & \ast & \ast & 0 \\
\ast & \ast & \ast & 0 \\
\ast & \ast & \ast & 0
\end{pmatrix}
\]Note that the entry  $\mathbf{r}^T\mathbf{v}\neq 0$ since $M\cdot (0,\mathbf{v})^T\neq 0$. Therefore, for $\mathbf{u}\in\mathbb{Z}^4$,
\begin{equation*}
    F(M_0\cdot\mathbf{u})=mu_1^2+u_1(\alpha u_2+\beta u_3+\gamma u_4)+Q(u_2,u_3),
\end{equation*}where $\gamma=\mathbf{r}^T\mathbf{v}\neq 0 $ and $Q$ is a binary quadratic form. Observe that there is no $u_4^2$ term. Hence, after the linear change of variables $mu_1+\alpha u_2+\beta u_3+\gamma u_4\to u_4$, we obtain, up to congruence restrictions,
\begin{equation*}
    \sum_{F(\mathbf{x})=0}A(x_1)w(\mathbf{x}/X)
    \approx
    \sum_{Q(u_2,u_3)=u_1u_4}A(u_1)V(\mathbf{u}/X),
\end{equation*}
where $V$ is a $1$-inert weight function obtained from $w$ through the above change of variables. One can then proceed by studying the congruence $Q(u_2,u_3)=0\pmod{u_1}$ for fixed $u_1$, which is a  well-studied problem.
\end{enumerate}

\subsection*{Notations}
In this paper the notation $\alpha\ll A$ will mean that for any $\epsilon>0$, there is a constant $c$ such that $|\alpha|\leq cAX^{\epsilon}$. The dependence of the constant on $\epsilon$, when occurring, will be ignored. We will follow the usual $\epsilon$-convention : the letter $\epsilon$ denotes sufficiently small positive quantity that may change from line to line.  We will also use the phrase ``negligible error'' by which we mean an error term $O_{B}(X^{-B})$ for an arbitrary $B>0$.

\subsection*{Acknowledgements}This work has benefited from the stimulating research environments of several institutions where parts of the paper were completed. All three authors would like to thank the Max Planck Institute for Mathematics, Bonn, Germany, where the work began, for providing excellent working conditions. The first author thanks the Indian Institute of Technology, Bombay, India for its excellent working conditions. The second and third authors would also like to thank the Lodha Mathematical Sciences Institute, Mumbai, India, where parts of the work were carried out, for providing an excellent research environment, and the Indian Institute of Technology Kanpur, where the work was finalised. The authors would like to thank Efthymios Sofos for suggesting the application to the Loughran--Smeets problem presented in Theorem \ref{conicfibres}. The second author is supported by ANRF under the Ramanujan Fellowship (RJF/2025/000357).

\subsection*{AI declaration} No artificial intelligence tools were used in this work and the paper is entirely the work of the authors.

\section{Outline of the proof}
\subsection{Initial reduction}
We are interested in the asymptotic of
\begin{equation*}
S=\sum_{\substack{\mathbf{x}\sim X, \\ F(\mathbf{x})=0}}A(x_1).
\end{equation*}
Our basic tool for detecting the equation $F(\mathbf{x})=0$ is the DFI delta symbol, which roughly gives
\begin{equation*}
S\approx \frac{1}{Q}\sum_{q\sim Q}\frac{1}{q}\,\,
\sideset{}{^*}\sum_{a\bmod q}\,\,
\sum_{\mathbf{x}\sim X}A(x_1)e\left(\frac{aF(\mathbf{x})}{q}\right)I_q(\mathbf{x}),
\end{equation*}
where $I_q(\mathbf{x})$ is smooth with oscillation of size $X^2/(qQ)$. We choose $Q=X$, so that $I_q(\mathbf{x})$ is essentially flat in the generic range $q\sim Q$. For the purposes of this sketch, we restrict attention to this generic situation and suppress the weight $I_q(\mathbf{x})$ throughout the remainder of the outline. At this stage, a trivial estimate gives $S\ll X^3\|A\|$. We therefore need to save $X^2$ plus something more in the above sum.

The next step is to apply the Poisson summation formula in the three free variables $x_2,x_3$, and $x_4$. Each sum has length $X$, which is same as the conductor $Q$. Consequently, the dual variables are essentially bounded by $1$, and we obtain, ignoring normalizing factors,
\begin{equation}\label{sketchpoisson}
S\approx \sum_{|\mathbf{c}'|\ll 1}\sum_{q\sim Q}\sum_{x_1\sim X}A(x_1)S_q(x_1,\mathbf{c}'),
\end{equation}where $\textbf{c}'=(c_2,c_3,c_4)$ and
\begin{equation*}
    S_q(x_1,\textbf{c}')=\sideset{}{^*}\sum_{a\modq}\,\sum_{\textbf{b}\in (\mathbb{Z}/q\mathbb{Z})^3}e\l(\frac{aF(x_1,\textbf{b})+\textbf{b}\cdot\textbf{c}'}{q}\r)
\end{equation*}
In each Poisson, we save $X^{1/2}$, while the sum over $a\bmod q$ provides an additional saving of $X^{1/2}$. The total saving is therefore $X^2$, which means we are at the boundary, so any further cancellation in \eqref{sketchpoisson} would suffice. 
\subsection{The main term}
As usual in this type of analysis, the main term arises from $\mathbf{c}'=(0,0,0)$, whose contribution is
\begin{equation*}
    \sum_{x_1}A(x_1)\sum_{q\sim Q}S_q(x_1,\mathbf{0}).
\end{equation*}
The sum over $q$ can be evaluated explicitly and leads to an expression involving the product of local densities $\prod_p \sigma_p(x_1)$, where
\begin{equation*}
    \sigma_p(x)=\lim_{\nu\to\infty} p^{-2\nu}\#\left\{ \mathbf{b}\bmod p^\nu: F(x,\mathbf{b})\equiv0\pmod{p^\nu} \right\}.
\end{equation*}As a function of $x_1$, the product $\prod_p \sigma_p(x_1)$ satisfies a useful factorisation property (see \eqref{factorisation}), which allows us to evaluate $\sum_{x_1}A(x_1)\prod_p \sigma_p(x_1)$ for many sequences $A(x_1)$ of interest. Some examples were highlighted above.
\subsection{The error term.} For $\mathbf{c}'\neq\mathbf{0}$, as indicated earlier, we are at the boundary and need to obtain some cancellation in the sum
\begin{equation*}
    \sum_{q\sim Q}\sum_{x_1\sim X}A(x_1)S_q(x_1,\mathbf{c}').
\end{equation*}
Since the sequence $A(x_1)$ is arbitrary, we cannot exploit any arithmetic structure in the $x_1$-sum. However, it turns out that we can obtain cancellation in
\begin{equation*}
    \sum_{q\sim Q}S_q(x_1,\mathbf{c}'),
\end{equation*}which is the main input of the paper. The starting point is to observe that although the character sum $S_q(x_1,\mathbf{c}')$ itself looks complicated, its Fourier coefficients modulo $q$ in the variable $x_1$ have a simple structure: they are essentially twisted Ramanujan sums. Indeed, orthogonality gives
\begin{equation}\label{expansion}
    S_q(x_1,\textbf{c}')=\frac{1}{q}\sum_{c_1\pmod q}\mathfrak{S}_q(\textbf{c})e\l(\frac{-x_1c_1}{q}\r),
\end{equation}where $\textbf{c}=(c_1,c_2,c_3,c_4)$ and
\begin{equation*}
    \mathfrak{S}_q(\textbf{c})=\sideset{}{^*}\sum_{a\pmod q}\,\,\sum_{\textbf{b}\in (\mathbb{Z}/q\mathbb{Z})^4}e\l(\frac{aF(\textbf{b})+\textbf{b}\cdot\textbf{c}}{q}\r).
\end{equation*}Note that the expansion \eqref{expansion} incurs a loss of a factor of $q$, since it introduces two sums over $c_1\bmod q$ and $b_1\bmod q$, together with a normalising factor of $1/q$. Square-root cancellation in each of these sums would recover this loss, bringing us back to the threshold. As indicated above, the crucial observation is that the character sum $\mathfrak{S}_q(\mathbf{c})$ is essentially a Ramanujan sum. This is easily seen for a diagonal form $F(\mathbf{b})=\sum_{i=1}^4m_ib_i^2$: each sum over $b_i\pmod q$ is a quadratic Gauss sum, which evaluates to
\[
\sqrt{q}\,\varepsilon_q\left(\frac{am_i}{q}\right)
e\left(-\frac{\overline{4am_i}c_i^2}{q}\right).
\]
Hence, the fourfold sum over $\mathbf{b}\in(\mathbb{Z}/q\mathbb{Z})^4$ produces the factor $\left(\frac{a}{q}\right)^4=1$, leaving the Ramanujan sum
\begin{equation*}
\sideset{}{^*}\sum_{a\pmod q}
e\left(\frac{\overline{a}F^*(\mathbf{c})}{q}\right),
\end{equation*}where $F^*$ is the dual form of $F$. This argument extends to a general form $F$ by diagonalising it modulo $q$. Indeed, for $(q,2\det M)=1$, one has
\begin{equation}\label{sketcheval}
    \mathfrak{S}_q(\mathbf{c}) =q^2\chi_M(q)c_q(F^*(\mathbf{c}))=q^2\chi_M(q) \sum_{d\mid(q,F^*(\mathbf{c}))} d\mu(q/d). 
\end{equation}
Here $\chi_M=\left(\frac{\det M}{\cdot}\right)$, $c_q$ denotes the Ramanujan sum modulo $q$. Note that $\mathfrak{S}_q(\textbf{c})$ is the same character sum as the one appearing in Heath-Brown's work \cite{HB}. For a general modulus, we write $q=q_1q_2$, where $q_1\mid(2\det M)^\infty$ and $(q_2,2\det M)=1$. A trivial estimate shows that the contribution when $q_1$ has any size goes into error. We then use multiplicativity $\mathfrak{S}_{q_1q_2}=\mathfrak{S}_{q_1}\mathfrak{S}_{q_2}, (q_1,q_2)=1$
to separate the bad-prime factor and focus primarily on moduli coprime to $2\det M$ and use their evaluation above. We consider the two extreme cases $d=1$ and $d=q$, which illustrate the main features of the intermediate cases.  The contribution from $d=1$ in \eqref{sketcheval} is independent of $c_1$. So in this case, after factoring out the bad prime factors, the $c_1$-sum in \eqref{expansion} reduces to

\begin{equation*}
\sum_{c_1\pmod q}e\left(-\frac{x_1\overline{k}c_1}{q}\right),
\end{equation*}
where $k$ denotes the small bad prime factor. We thus gain a factor $q$ due to orthogonality. So we can focus on the main case $d=q$, in which case its contribution towards $\sum_{q\sim Q}S_q(x_1,\textbf{c}')$ has the form
\begin{equation*}
    \sum_{\substack{q\sim Q\\}}\chi_M(q)\sum_{\substack{c_1\pmod q\\F^*(\textbf{c})=0\pmod q}}e\l(\frac{-x_1\overline{k}c_1}{q}\r).
\end{equation*}
These are essentially Weyl sums for quadratic roots and their estimation depends on the nature of the quadratic polynomial $F^{*}(\textbf{c})$ is $c_1$. 
\subsubsection{The irreducible case : equidistribution of quadratic roots}
When $F^*(x,\mathbf{c}')$ is an irreducible quadratic polynomial in $x$, in absense of the factor $\overline{k}$, the Weyl sums above have been studied in the important works of Hooley \cite{Hoo} and Duke--Friedlander--Iwaniec \cite{DFIr}, which establish power savings. These results can however cannot be invoked directly because of the presence of $\overline{k}$. The authors are not aware of results in the current literature that cover this generalization. In principle, the works cited above could be extended to this setting by incorporating a nontrivial level into the analysis. To keep the exposition simple, we instead exploit the $x_1$-variable and consider the bilinear sum
\begin{equation*}
\sum_{q\sim Q} \sum_{\substack{c_1\bmod q\\F^*(\mathbf{c})\equiv0\pmod q}}\sum_{x_1\sim X} A(x_1)e\left(-\frac{x_1\overline{k}c_1}{q}\right).
\end{equation*}Fixing residue classes $x_1=\beta\pmod k$, the above is then bounded by
\begin{equation*}
\sum_{q\sim Q}\sum_{\substack{c_1\bmod q\\F^*(\mathbf{c})\equiv0\pmod q}}\l| \sum_{\lambda\sim X/k} A(k\lambda+\beta)e\l(-\frac{\lambda c_1}{q}\r)\r|.
\end{equation*}
We then apply the quadratic-root large sieve from \cite{BBDT}. This approach exploits the length of the $x_1$-sum without requiring any arithmetic information about the coefficients $A(x_1)$.
\subsubsection{The reducible case : Weil's bound for Kloosterman sums}
When $F^*(x,\mathbf{c}')$ is reducible, say
\begin{equation*}
    F^*(x,\mathbf{c}')=(x-\alpha_1)(x-\alpha_2),
\end{equation*}
we may, roughly speaking, decompose $q=d_1d_2,$ so that the condition $F^*(\mathbf{c})\equiv0\pmod q$ becomes
\begin{equation*}
    c_1\equiv\alpha_1\pmod{d_1}, \qquad \textrm{and} \ \ \ \ \ c_1\equiv\alpha_2\pmod{d_2}.
\end{equation*}
We assume $(d_1,d_2)=1$ and $d_1\leq d_2$ for the purpose of this sketch. The Chinese remainder theorem and additive reciprocity then give
\begin{align*}
e\left(-\frac{x_1\overline{k}c_1}{q}\right)=e\left( -\frac{x_1\overline{k}\overline{d_2}\alpha_1}{d_1}-\frac{x_1\overline{k}\overline{d_1}\alpha_2}{d_2}\right)=e\left(-\frac{x_1\overline{d_2}(k\overline{k}\alpha_1-\alpha_2)}{kd_1}\right)
 e\left(-\frac{x_1\alpha_2}{kd_1d_2}\right),
\end{align*}
where $\overline{k}$ is the inverse modulo $d_1$. The final factor is essentially flat in the ranges under consideration and can be absorbed into the weight functions. Hence the contribution of this part towards $\sum_{q\sim Q}S_q(x_1,\textbf{c}')$ is roughly
\begin{equation}\label{reduciblecharacter}
\mathop{\sum\sum}_{\substack{d_1d_2\sim Q\\d_1\leq d_2}}\chi_M(d_1d_2) e\left(-\frac{x_1\overline{d_2}(k\overline{k}\alpha_1-\alpha_2)}{kd_1}\right).
\end{equation}
The next step is to apply Poisson summation to the $d_2$-sum. Let us again consider the two extreme cases $d_1\sim 1$ and $d_1\sim d_2$, which encapsulate the arguments used for all intermediate ranges. 

When $d_1\sim 1$, only the zero frequency contributes, and orthogonality of $\chi_M$ forces $k|\det M$. In this case, we obtain a Gauss sum modulo $|\det M|$, which produces the factor $\chi_M(x_1)$. We then assume cancellation in the remaining $x_1$-sum,
\begin{equation*}
    \sum_{x_1}A(x_1)\chi_M(x_1).
\end{equation*}
If the coefficients $A(x_1)$ correlate with $\chi_M(x_1)$, for example, when $A$ is supported on squares, this contribution may instead give rise to an additional main term. One could formulate a modified theorem by extracting this term, but we omit this extension to keep the exposition simple.

Finally, when $d_1\sim d_2\sim Q^{1/2}$, Poisson summation essentially transforms the $d_2$-sum in \eqref{reduciblecharacter} into a Kloosterman sum modulo $d_1$, for which Weil's bound yields square-root cancellation.

\section{Preliminaries}
\subsection{The delta method.} We  use a version of the delta method due to Duke, Friedlander and Iwaniec. More specifically, we will use the expansion $(20.157)$ given in Chapter 20 of \cite{IK}. Let $\delta : \mathbb{Z}\to \{0,1\}$ be defined by
\begin{align*}
\delta (n)=
\begin{cases}
1 &\text{if}\,\,n=0 \\
0 &\text{otherwise}.
\end{cases}
\end{align*}
Then for $n\in\mathbb{Z}\cap [-2M,2M]$, we have
\begin{equation}\label{deltasymbol}
\delta (n)=\frac{1}{Q}\sum_{q\leq Q}\frac{1}{q}\sum_{a\bmod q}e\left(\frac{na}{q}\right)\int_{\mathbb{R}}g(q,x)e\left(\frac{nx}{q	Q}\right) \text{dx},
\end{equation} 
where $Q=2M^{1/2}$ and the function $g$ satisfies the following lemma.
\begin{lemma}
We have
\begin{equation}\label{delta}
\begin{aligned}
&g(q,x)=1+O\left(\frac{Q}{q}\left(\frac{q}{Q}+|x|\right)^A\right),\,\,\,\, \,\,\,\, g(q,x)\ll |x|^{-A},\\
&x^mq^n\frac{\partial^{m+n}}{\partial x^m\partial q^n} g(q,x)\ll \log Q\min \{Q/q,1/|x|\},
\end{aligned}
\end{equation}for any $A>1$ and $m,n \geq 0$.
\end{lemma}
\begin{proof}
For the first two property see $(20.158)$ and $(20.159)$ of \cite{IK}. For the third property, recall that
\begin{equation*}
        g(q,x)=\int_{\mathbb{R}}\Delta_q(u)h(u)e\l(\frac{ux}{qQ}\r) \text{du}.
\end{equation*}
\end{proof}
In particular, the second property implies that the effective range of the integral in \eqref{deltasymbol} is $[-M^{\epsilon},M^{\epsilon}]$. It also follows from the first property that if $\min \{Q/q,1/|x|\} > Q^{\epsilon}$, we can replace $g$ by $1$ up to a negligible error term, and if $\min \{Q/q,1/|x|\} \leq Q^{\epsilon}$, then the third property implies $x^mq^n\frac{\partial^{m+n}}{\partial x^m\partial q^n} g(q,x)\ll Q^{\epsilon}$. In any case, we can view $g(q,x)$ as a `nice' weight function in both of the variables $x$ and $q$.

We will use the following version of stationary phase analysis from \cite{KPY}, which builds upon the work of Blomer-Khan-Young \cite{bky}. The theorem uses the language of inert functions, which is defined as follows.

\begin{definition}\label{inert}
A family $\{w_T\}_{T\in \mathcal{F}}$ of smooth  
functions supported on a product of dyadic intervals in $\R_{>0}^d$ is called $X$-inert if for each $j=(j_1,\ldots,j_d) \in \Z_{\geq 0}^d$ we have \begin{equation*}
C(j_1,\ldots,j_d) 
:= \sup_{T \in \mathcal{F}} \sup_{(x_1, \ldots, x_d) \in \R_{>0}^d} X_T^{-j_1- \cdots -j_d}\left| x_{1}^{j_1} \cdots x_{d}^{j_d} w_T^{(j_1,\ldots,j_d)}(x_1,\ldots,x_d) \right| < \infty.
\end{equation*}
\end{definition}
The notion of $X$-inert measures the uniformity of the ``flatness'' of the functions $w_T$ as we move across the family $\mathcal{F}$. 

\begin{lemma}\label{thm:exponentialintegral}
 Suppose $w_T$ is $X$-inert in $t_1, \dots t_d$, supported on $t_1 \asymp Z$ and $t_i \asymp X_i$ for $i=2,\dots, d$.  Suppose that on the support of $w_T$, $\phi = \phi_T$ satisfies
\begin{equation} 
 \frac{\partial^{a_1 + a_2 + \dots + a_d}}{\partial t_1^{a_1} \dots \partial t_d^{a_d}} \phi(t_1, t_2, \dots, t_d) \ll_{C_\mathcal{F}} \frac{Y}{Z^{a_1}} \frac{1}{X_2^{a_2} \dots X_d^{a_d}}, \notag
\end{equation}
for all $a_1, \ldots, a_d \in \N$. Suppose $ \phi''(t_1, t_2, \dots, t_d) \gg \frac{Y}{Z^2}$, (here and below, $\phi'$  and $\phi''$ denote the derivative with respect to $t_1$) for all $t_1, t_2, \dots, t_d$ in the support of $w_T$, and there exists $t_0 \in \mathbb{R}$ such that $\phi'(t_0) = 0$  (note $t_0$ is necessarily unique). Suppose that $Y/X^2 \geq R \geq 1$.  Then 

\begin{equation} 
I = \int_{\mathbb{R}} e^{i \phi(t_1, \dots, t_d)} w_T(t_1, \dots, t_d) dt_1 =  \frac{Z}{\sqrt{Y}} e^{i \phi(t_0, t_2, \dots, t_d)} W_T(t_2, \dots, t_d)
 + O_{A}(ZR^{-A}), \notag
\end{equation}
for some $X$-inert family of functions $W_T$, and where $A>0$ may be taken to be arbitrarily large. The implied constant in above equation depends only on $A$ and $C_{\mathcal{F}}$.
\end{lemma}

A crucial input in this paper comes from the equidistribution of quadratic roots, or more precisely, from cancellation in exponential sums over the roots of a fixed irreducible polynomial. This phenomenon has been extensively studied since the work of Hooley \cite{Hoo}. For our purposes, however, the most convenient formulation is provided by the following quadratic-root large sieve due to Balog--Blomer--Dartyge--Tenenbaum (see \cite[p.666, Section 5.3]{BBDT}).
\begin{lemma} \label{quadractic-rootlargeseive}
Let $ F(X,Y)=aX^2+bXY+cY^2\in\mathbb Z[X,Y]$ be a fixed binary quadratic form whose discriminant $ \Delta_F=b^2-4ac$ is not a perfect square. Then, for any $Q,N\geq 1$ and any sequence $(\lambda_n)_{n\leq N}$ of complex numbers, one has
\begin{align*}
\sum_{Q<d\leq 2Q}\ \sum_{\substack{\nu\,(\mathrm{mod}\,d)\\ F(\nu,1)\equiv 0\,(\mathrm{mod}\,d)}} \left|\sum_{n\leq N} \lambda_n e\left(\frac{n\nu}{d}\right)\right|^2\ll_F (Q+N)\sum_{n\leq N}|\lambda_n|^2.
\end{align*}
The implied constant depends only on $F$.
\end{lemma}

\begin{lemma}\label{gauss-sum}
Let $\chi$ be a primitive non-principal Dirichlet character modulo
$q>1$. Let $R,\lambda\geq1$ and $\alpha\in\mathbb Z$. Then
\begin{equation}\notag
\sum_{d\bmod {[R,\lambda q]}}\chi(d)e\!\left(\frac{\alpha d}{R}\right) =
\begin{cases}
\frac{[R,\lambda q]}{q}\,\,\overline{\chi}(q\alpha/R)\,\tau(\chi),
& R\mid q\alpha,\\
0,&R\nmid q\alpha.
\end{cases}
\end{equation}
\end{lemma}

\begin{proof}
Write $d=u+qv$, with $u\bmod q$ and $v\bmod [R,\lambda q]/q$. Then
\begin{equation}\notag
\sum_{d\bmod [R,\lambda q]}\chi(d)e\!\left(\frac{\alpha d}{R}\right) =
\sum_{u\bmod q}\chi(u)e \left(\frac{\alpha u}{R}\right) \sum_{v\bmod [R,\lambda q]/q}e\left(\frac{\alpha qv}{R}\right).
\end{equation}
By orthogonality, the inner sum vanishes unless $R\mid q\alpha$, in which case it equals $[R,\lambda q]/q$. Setting $m=q\alpha/R$ and using the primitive Gauss-sum identity gives
\begin{equation}\notag
\frac{[R,\lambda q]}{q}\sum_{u\bmod q}\chi(u)e(mu/q) =\frac{[R,\lambda q]}{q} \overline{\chi(m)}\,\tau(\chi).
\end{equation}
This proves the formula.
\end{proof}

\section{Initial reduction: Separating the harmonics} 
The object of our study is
\begin{equation*}
S(X)=\sum_{\substack{\textbf{x}\in\mathbb{Z}^4\\F(\textbf{x})=0}}A(x_1) w\left(\frac{\textbf{x}}{X}\right),
\end{equation*}
where $\textbf{x}=(x_1,x_2,x_3,x_4)$. Throughout what follows, a boldface variable $\textbf{z}$ will always denote a vector in $\mathbb{R}^k$, where the dimension $k$ will be clear from the context. We detect the surface $F(\textbf{x})=0$ using DFI's delta symbol \eqref{deltasymbol}.  We get
\ba \label{Stilde}
S(X) &=\frac{1}{Q}\sum_{q\leq Q}\frac{1}{q}\,\,\,\,\sideset{}{^*}\sum_{a \modq} \sum_{\substack{\textbf{x}}} A(x_1)w\l(\frac{\textbf{x}}{X}\r)   e\left(\frac{aF(\textbf{x})}{q}\right) \int_{\R} g(q,y)e \l(\frac{F(\textbf{x})y}{qQ}\r)\diff y,
\ea
with the choice $  Q=X.$ Here the weight function $g(q,y)$ is as in \eqref{deltasymbol}.
We now apply the Poisson summation formula to free variables $x_2$, $x_3$ and $x_4$. We get  
\ba 
S(X)  &= \frac{1}{Q}\sum_{q\leq Q}\frac{1}{q^4} \sum_{x_1 \in \Z} A(x_1) \sum_{\textbf{c}' \in \Z^3} S_q(x_1,\textbf{c}') I_q(x_1,\textbf{c}'), \ \ \ \textrm{where} \notag
\ea 

\ba \label{initalcharatersum}
S_q(x_1,\textbf{c}') = \sideset{}{^*}\sum_{a\modq}\sum_{\textbf{u} \modq} e\l(\frac{aF(x_1, \textbf{u}) + \textbf{c}' \cdot \textbf{u}}{q}\r), \ \ \ \textrm{and}
\ea 
\ba \label{initalintegral}
I_q(x_1,\textbf{c}') = \int_{\R^3} \int_{\R} g(q,y) w\l(\frac{(x_1, \textbf{t})}{X}\r) e\l(\frac{F(x_1, \textbf{t})y}{qQ} - \frac{\textbf{c}' \cdot \textbf{t}}{q}\r) \diff y\,\, \diff\textbf{t}, 
\ea 
with $\textbf{c} '= (c_2,c_3,c_4)$.
We will see that the main term occurs when all three dual variables are zero, i.e., $c_2=c_3=c_4=0$, and the complementary case is absorbed into the error term. Hence, we write $S(X)$ as $ S(X) = M(X) + E(X), $ where

\ba
M(X) = \frac{1}{Q}\sum_{q\leq Q}\frac{1}{q^4} \sum_{x_1 \in \Z} A(x_1) S_q(x_1,\textbf{0}) I_q(x_1,\textbf{0}), \notag  \ \ \ \textrm{and}
\ea
 
\ba \label{errorterm0}
E(X) = \frac{1}{Q}\sum_{q\leq Q}\frac{1}{q^4}\sum_{x_1 \in \Z} A(x_1) \sum_{\substack{\textbf{c}' \in \Z^3\\ \textbf{c}' \neq \textbf{0}}} S_q(x_1,\textbf{c}') I_q(x_1,\textbf{c}'). 
\ea Note that, throughout the range of summation in \eqref{errorterm0}, we have $x_1\sim X$. Hence, we may choose a bump function $V$ supported in an interval in $\mathbb{R}_{>0}$ such that $V(x_1/X)=1$ for all points of summation in \eqref{errorterm0} and consider
\begin{equation}\label{errorterm}
    E(X) = \frac{1}{Q}\sum_{q\leq Q}\frac{1}{q^4}\sum_{x_1 \in \Z} A(x_1)V(x_1/X) \sum_{\substack{\textbf{c}' \in \Z^3\\ \textbf{c}' \neq \textbf{0}}} S_q(x_1,\textbf{c}') I_q(x_1,\textbf{c}')
\end{equation}


\section{Extracting the  main term from \texorpdfstring{$M(X)$}{M(X)}}

For the notational purpose, we replace $x_1$ with $n$.  We have 
\begin{equation}\label{maintermeval}
    M(X)=\frac{1}{Q}\sum_{q\leq Q}\frac{1}{q^4}\sum_{n \geq 1} A(n) S_q(n) I(n,q), \ \ \ \ \textrm{where}
\end{equation}
\begin{equation*}
    S_q(n)=S_q(n,\textbf{0})=\sideset{}{^*}\sum_{a\bmod q}\sum_{{\bf b} \bmod q}
e\!\left(\frac{aF(n,{\bf b})}{q}\right), \ \ \ \textrm{with} \ \ {\bf b} = (b_2, b_3, b_4)
\end{equation*}and 
\begin{equation}\label{Irecall}
    I(n,q)=I_q(n,\textbf{0})=\int_{\R^3}w\l(\frac{ (n,\textbf{t})}{X}\r) \int_{\R} g(q,y)  e\l(\frac{F(n, \textbf{t})y}{qQ} \r) \diff y\,\, \diff\textbf{t}.
\end{equation}
Our strategy is first to extract a preliminary main term by summing over $q$ with $n$ fixed. The resulting coefficient depends on $n$ and, as we shall see, possesses a convenient multiplicative structure. We then sum these coefficients against $A(n)$, exploiting the arithmetic properties of the sequence $A(n)$.

We begin by observing that, for fixed $n$, $S_q(n)$ behaves like $\chi_M(q)$ (see Proposition \ref{prop:good-prime-power}). Consequently, the sum over $q$ exhibits cancellation once the range of summation has some size, and only the contribution from small $q$ enters the main term. This is reflected in the following propositions and culminates in Lemma \ref{tail}.

\subsection{The singular series}
\begin{proposition}\label{prop:truncated-singular-series}
Suppose that $\det M$ is not a square. Then, for any $Y>0$,
\begin{equation}\label{eq:truncated-singular-series}
\sum_{q\leq Y}\frac{S_q(n)}{q^{3}}=\mathfrak{S}_{F}(n)+ O_{\varepsilon} \left(Y^{-1/8+\varepsilon}\right), \ \ \ \ \textrm{where}
\end{equation}

\begin{equation*} 
\mathfrak{S}_{F}(n) = \prod_{p}\sigma_{p}(n)  \ \ \ \ \textrm{with} 
\end{equation*}

\begin{equation}
\sigma_{p}(n):=\sum_{k=0}^{\infty} \frac{S_{p^{k}}(n)}{p^{3k}}=\lim_{\nu\to\infty}
p^{-2\nu}\#\left\{ \mathbf{b}\bmod p^{\nu}:F(n,\mathbf{b})\equiv 0\pmod{p^{\nu}}
\right\}.
\end{equation}
The product $\prod_p \sigma_p(n)$ is interpreted as $L(1,\chi_M)\prod_p(1-\chi_M(p)/p)\sigma_p(n)$, where the Euler product on the right converges absolutely. 
\end{proposition}

\begin{proof}

Consider the Dirichlet series
\begin{equation}\label{eq:dirichlet-series-Z}
Z_{n}(s)
=
\sum_{q=1}^{\infty}\frac{S_q(n)}{q^{s+3}},
\qquad
\Re(s)>1.
\end{equation}
Let $\mathcal{P}=\{p\,\,\text{prime}: p|2n\det(M)\det(N)\}$. Using the  multiplicativity and  the Proposition \ref{prop:good-prime-power},  we see that its Euler
factors outside \(\mathcal{P}\) are $ 1 +\chi_{M}(p)/ p^{1+s} $. 
Hence
\begin{equation}\label{eq:factorization-Z}
Z_{n}(s)
=
H_{n}(s)
\frac{L(s+1,\chi_{M})}
     {L(2s+2,\chi_{M}^{2})}, \qquad \textnormal{where}
\end{equation}

\begin{equation}\label{eq:bad-euler-product}
H_{n}(s)
=
\prod_{p\in\mathcal{P}}
\frac{
\displaystyle
\sum_{k=0}^{\infty}S_{p^{k}}(n)/p^{k(s+3)}
}{
1+\chi_{M}(p)p^{-1-s}
}.
\end{equation}Since we will using a contour shifting argument, we need upper bounds for $H_n(s)$ in vertical strips. To this end, using Lemma \ref{lem:Hn-bound}, we have that $ H_n(s)$ is holomorphic to the right of line $ \sigma = -1/2$ and 
\begin{equation}\label{finiteub}
H_n(\sigma+it)\ll_{\varepsilon} n^{\varepsilon} ,\,-1/2 + \epsilon\leq\sigma \leq 0. 
\end{equation} Since \(\det(M)\) is not a square, \(\chi_{M}\) is non-principal, and hence
\(L(s+1,\chi_{M})\) is entire. Moreover, if $\Re(s)\geq -1/2+\varepsilon$, then $\Re(2s+2)\geq 1+2\varepsilon$. Consequently, $L(2s+2,\chi_{M}^{2})\neq 0$ throughout this region. Thus \(Z_{n}(s)\) is holomorphic in
\begin{equation}\label{eq:holomorphic-region}
\Re(s)\geq -\frac{1}{2}+\varepsilon.
\end{equation}The local prime-power estimates from the Proposition \ref{prop:bad-prime-bound} imply that $ S_q(n)\ll_{\varepsilon} q^{5/2+\varepsilon}.$ Let $T>0$ be large which we will chose later. By the truncated Perron formula,
\begin{equation}\label{eq:perron-formula}
\sum_{q\leq Y}\frac{S_q(n)}{q^3}
=
\frac{1}{2\pi i}
\int_{1+\varepsilon-iT}^{1+\varepsilon+iT}
Z_{n}(s)\frac{Y^{s}}{s}\,ds
+
O_{\varepsilon}(Y^{1+\varepsilon}T^{-1}).
\end{equation}We move the contour in \eqref{eq:perron-formula} to $ \Re(s)=-1/2 +\varepsilon.$ The only pole crossed is the simple pole of \(1/s\) at \(s=0\), whose
residue is \(Z_{n}(0)\). Hence
\begin{equation}\label{eq:contour-shift}
\sum_{q\leq Y}\frac{S_q(n)}{q^3}
=
Z_{n}(0)
+
\frac{1}{2\pi i}
\int_{-\frac12+\varepsilon-iT}
     ^{-\frac12+\varepsilon+iT}
Z_{n}(s)\frac{Y^{s}}{s}\,ds
+
E_{\mathrm{hor}}
+
O_{\varepsilon}(YT^{-1}),
\end{equation}
where \(E_{\mathrm{hor}}\) denotes the contribution of the horizontal
segments. On the new vertical line, \eqref{eq:factorization-Z} and \eqref{finiteub} gives
\begin{equation}\label{eq:Z-on-critical-line}
\begin{aligned}
\left|Z_{n}\left(-\frac12+\varepsilon+it\right)\right|
&=
\left|H_{n}\left(-\frac12+\varepsilon+it\right)\right|
\frac{
\left|L\left(\frac12+\varepsilon+it,\chi_{M}\right)\right|
}{
\left|L\left(1+2\varepsilon+2it,\chi_{M}^{2}\right)\right|
}\\
&\ll_{\varepsilon}n^{\varepsilon}\left|L\left(\frac12+\varepsilon+it,\chi_{M}\right)\right|
\end{aligned}
\end{equation}Using the convexity bound  $ \left|L\left(1/2+\varepsilon+it,\chi_{M}\right)\right|\ll_{M,\varepsilon} T^{1/4+\varepsilon},$ for $ t \leq T$, we get 
\begin{equation}\label{eq:vertical-integral-bound}
\int_{-\frac12+\varepsilon-iT}^{-\frac12+\varepsilon+iT}
Z_{n}(s)\frac{Y^{s}}{s}\,ds \ll_{\varepsilon} T^{1/4}Y^{-1/2+\varepsilon}.
\end{equation}For the $E_{\mathrm{hor}}$ contribution, again from the convexity bound and \eqref{finiteub} we know $ |Z_n(\sigma+it)|\ll_{\varepsilon}n^{\varepsilon}T^{-\sigma/2}$ for $ -1/2\leq \sigma\leq 0$. Hence we get
\begin{equation}\label{eq:horizontal-integral-bound}
E_{\mathrm{hor}}\ll_{\varepsilon} n^{\varepsilon}T^{1/4} (Y/T)
\end{equation}
Substituting \eqref{eq:vertical-integral-bound} and
\eqref{eq:horizontal-integral-bound} into \eqref{eq:contour-shift}, and choosing the optimal value $T=Y^{3/2}$, we obtain
\begin{equation}\label{eq:partial-sum-Z0}
\sum_{q\leq Y}\frac{S_q(n)}{q^3}
=
Z_{n}(0)+
O_{\varepsilon}\left(Y^{-1/8+\varepsilon}\right).
\end{equation}

It remains to relate \(Z_{n}(0)\) to the product of $p$-adic densities. Let
\label{padicdensity3var}\[
N_{p^{\nu}}(n)
=
\#\left\{
\mathbf{b}\bmod p^{\nu}:
F(n,\mathbf{b})\equiv0\pmod{p^{\nu}}
\right\}.
\]
Detecting the congruence using additive characters and using the decomposition 
\begin{align*}
  \sum_{a\bmod p^{\nu}}e(a\alpha/p^{\nu})=\sum_{k=0}^{\nu}\,\,\,\sideset{}{^*}\sum_{b\bmod p^{k}}e(b\alpha/p^{k}) \Rightarrow   p^{-2\nu}N_{p^{\nu}}(n)=\sum_{k=0}^{\nu} \frac{S_{p^{k}}(n)}{p^{3k}}.
\end{align*}

The proof is complete after passing to the limit \(\nu\to\infty\).
\end{proof}

\begin{proposition}\label{tailbd}
    For $n\geq 1$, we have
    \begin{equation*}
        \sum_{Q^{1-\varepsilon}\leq q\leq Q}\frac{S_q(n)I(n,q)}{q^4}\ll Q^{-1-1/8+\varepsilon}\int_{\R^3}w\l(\frac{ (n,\textbf{t})}{X}\r)\,d\textbf{t}.
    \end{equation*}
\end{proposition}
\begin{proof}
    We will use summation by parts and invoke Proposition \ref{prop:truncated-singular-series}. To do that, we need to know the derivative bounds of $I(n,q)/q$. To this end, note that from the definition in \cite{IK}
    \begin{equation}\label{deltatransform}
        \int_{\mathbb{R}}g(q,y)e\left(\frac{my}{qQ}\right)\,dy=qQ\,\Delta_q(m),
    \end{equation}where as in \cite{IK}
    \begin{equation}\label{deltadef}
        \Delta_q(u)=\sum_{r=1}^{\infty}(qr)^{-1}(h(qr)-h(|u|/qr)),
    \end{equation}where $h(u)$ can be chosen to be a 1-inert function supported on $u\sim Q$ and normalized such that $\sum_{q=1}^{\infty}h(q)=1$. Note that the series above vanishes unless $r\ll Q/q+|u|/(qQ)$. It is proven in \cite{IK} that for such a function $h$, we have
    \begin{equation}\label{deltaub}
        \Delta_q(u)\ll (Q^2+qQ)^{-1}+(|u|+qQ)^{-1}.
    \end{equation}We claim that a similar bound holds for derivative $\frac{\partial}{\partial q}\Delta_q(u)$. First note that
    \begin{align*}
         q \frac{\partial}{\partial q}\Delta_q(u) = \sum_{r=1}^{\infty}(qr)^{-1}(h(qr)-h(|u|/qr))+\sum_{r=1}^{\infty}(qr)^{-1}(qrh(qr)-(|u|/qr)h(|u|/qr)).
    \end{align*}The two sums inside the parenthesis above is in the same format as \eqref{deltadef}. Hence the same argument as in \eqref{deltaub} gives
    \begin{equation*}
        \frac{\partial}{\partial q}\Delta_q(u)\ll q^{-1}(Q^2+qQ)^{-1}+q^{-1}(|u|+qQ)^{-1}.
    \end{equation*}Consequently
        \begin{equation*}
            \frac{\partial}{\partial q}\left(\frac{1}{q}\int_{\mathbb{R}}g(q,y)e\left(\frac{my}{qQ}\right)\,dy\right)\ll (Q/q)\left((Q^2+qQ)^{-1}+(|m|+qQ)^{-1}\right) .
        \end{equation*}For us $|m|=|F(n,\textbf{t})|\ll X^2$ and $Q=X$, so the bound in the right hand side of the last equation is $\ll (Q/q)(qQ)^{-1}=q^{-2}$. Substituting this bound in \eqref{Irecall} after taking the derivative w.r.t $q$ we get
        \begin{equation}\label{Iderbd}
            \frac{\partial}{\partial q}I(n,q)/q\ll_a q^{-2}\int_{\R^3}w\l(\frac{ (n,\textbf{t})}{X}\r)\,d\textbf{t}.
        \end{equation}Using partial summation, Proposition \ref{prop:truncated-singular-series} and \eqref{Iderbd} we then get
        \begin{equation*}
            \sum_{Q^{1-\varepsilon}\leq q\leq Q}\frac{S_q(n)I(n,q)}{q^4}\ll Q^{-1-1/8+\varepsilon}\int_{\R^3}w\l(\frac{ (n,\textbf{t})}{X}\r)\,d\textbf{t}.
        \end{equation*}
    \end{proof}
     Using Proposition \ref{tailbd} and bounding everything else trivially, we conclude that
    \begin{equation}
         \frac{1}{Q}\sum_{Q^{1-\varepsilon}\leq q\leq Q}\frac{1}{q^4}\sum_{n \geq 1} A(n) S_q(n) I(n,q)\ll X^{3}Q^{-2-1/8+\varepsilon}\|A\|  \notag . 
    \end{equation} Substituting the last bound for the tail part in \eqref{maintermeval},  we get
    \begin{lemma}\label{tail}
    We have 
    \begin{equation}\label{tailseparated}
M(X)=\frac{1}{Q}\sum_{n \geq 1}\sum_{q\leq Q^{1\varepsilon}} \frac{1}{q^4} A(n) S_q(n) I(n,q)+O(X^{3}Q^{-2-1/8+\varepsilon}\|A\|).
    \end{equation}
\end{lemma}

       When the modulus $q$ is somewhat smaller than its generic size $Q$, the quantity $I(n,q)/q$ can essentially be replaced by a constant with respect to $q$, related to a `singular' integral.
    \subsection{The singular integral}
    \begin{proposition}\label{presingularint}
   Denote $G(\textbf{t})=F(n/X,\textbf{t})$. Suppose that $|\nabla G|\gg 1 $ on $\text{supp}(w(n/X,\textbf{t}))$. Then for $q\leq Q$ we have
    \begin{equation*}
       I(n,q)/q=X^2\sigma_{\infty}(n/X)+O_A\left(X^2(q/Q)^{A}\right),
    \end{equation*}where
    \begin{equation*}
        \sigma_{\infty}(a)=\frac{1}{2\varepsilon}\lim_{\varepsilon\to 0}\int_{|F(a,\textbf{t})|\leq \varepsilon}w(a,\textbf{t})\,d\textbf{t}.
    \end{equation*}
\end{proposition}Furthermore, $ \sigma_\infty (x)$ is smooth  and $ \sigma_\infty^{(j)}(x)\ll_{F,w,j} 1 $,  for every \(j\ge0\). 
\begin{proof}Let us first prove the claim for weight functions $w$ satisfying $\frac{\partial G}{\partial t_1}\gg 1$ for all $\textbf{t}\in \text{supp}(w(n/X,\textbf{t}))$. In this case, we can substitute $F(n/X,\textbf{t})=z$ for the $t_1$-integral in \eqref{Irecall} to get
\begin{equation*}
    I(n,q)=X^3\int J(z,n/X)\int g(q,y)e(X^2zy/qQ)dy\,dz, \ \ \ \ \textrm{where}
\end{equation*}
\begin{equation*}
    J(z,n/X)=\int_{\mathbb{R}^2}w(n/X,t_1(z),t_2,t_3) \frac{d t_2dt_3}{\frac{\partial G}{dt_1}}
\end{equation*}where $t_1(z)$ is defined by the equation $G(\textbf{t})=z$. Note that $ \partial t_1(z)/ \partial z = \left(\partial G/  \partial t_1\right)^{-1} \ll 1$ and consequently all higher derivatives 
\begin{equation}\label{ider}
    \frac{\partial^{j_1+j_2}J(z,x)}{\partial z^{j_1}\partial x^{j_2}}\ll_{j_1,j_2} 1.
\end{equation}
Now as in \eqref{deltatransform} we know
\begin{equation*}
   \int g(q,y)e(X^2zy/qQ)dy=qQ\Delta(X^2z) ,
\end{equation*}where $\Delta$ as in \eqref{deltadef}. Hence
\begin{equation*}
    I(n,q)/q=XQ\int_{\mathbb{R}}J(z/X^2, n/X)\Delta(z)dz.
\end{equation*}Now Corollary 20.19 in \cite{IK} says that for any compactly supported smooth function $f$,
\begin{equation*}
    \int_{\mathbb{R}}\Delta_q(z)f(z)\,dz=f(0)+E_q(f), \ \ \ \textrm{where}
\end{equation*} 
\begin{equation*}
    E_q(f)\ll_A (qQ)^{A-1}\int_{\mathbb{R}}(Q^{-2A}|f(z)|+|f^{(A)}(z)|)\,dz.
\end{equation*}For $f(z)=J(z/X^2, n/X)$ using \eqref{ider} we have
\begin{equation}
    E_q(f)\ll_A (qQ)^{A-1}(Q^{-2A}+X^{-2A})\ll (qQ)^{-1}(q/Q)^{A}
\end{equation}since $Q=X$ by our choice. Hence we arrive at 
\begin{equation*}
    I(n,q)/q=XQJ(0,n/X)+(X/q)O_A((q/Q)^A)).
\end{equation*}Here
\begin{equation}\label{JI}
    J(0,n/X)=\int_{\mathbb{R}^2}w(n/X,t_1,t_2,t_3) \frac{d t_2dt_3}{\frac{\partial G}{dt_1}}
\end{equation}where $t_1$ is defined by the equation $G(\textbf{t})=0$. Since by assumption we have $\frac{\partial G}{\partial t_1}\gg 1$ on the support of $w(n/X,\textbf{t})$, we invoke Lemma 11 of Heath-Brown \cite{HB} to see that 
\begin{equation*}
    J(0,n/X)=\frac{1}{2\varepsilon}\int_{|F(n/X,\textbf{t})|\leq \varepsilon}w(n/X,\textbf{t})\,d\textbf{t}+O(\varepsilon).
\end{equation*}The claim follows.  We can now extend the condition $\partial G/\partial t_1\gg 1$ to $\nabla G\gg 1$ using the localization lemma given in \cite[Lemma 12]{HB}. 

For the second part, note that for weight functions $w$ satisfying $\frac{\partial G}{\partial t_1}\gg 1$ for all $\textbf{t}\in \text{supp}(w(n/X,\textbf{t}))$, the relation \eqref{JI} explicitly gives
\begin{equation}\label{localint}
    \sigma_{\infty}(x)=\lim_{\varepsilon\to 0}\frac{1}{2\varepsilon}\int_{|F(n/X,\textbf{t})|\leq \varepsilon}w(x,\textbf{t})\,d\textbf{t}=\int_{\mathbb{R}^2}w(x,t_1,t_2,t_3) \frac{d t_2dt_3}{\frac{\partial G}{dt_1}},
\end{equation}in which case from \eqref{ider} it follows that $\sigma_\infty^{(j)}(x)\ll_{F,w,j} 1 $. For a general weight $w$, we may once again apply \cite[Lemma 12]{HB} to decompose the integral $\sigma_{\infty}(x)$ into a sum of integrals of the form appearing in \eqref{localint}.
\end{proof} 

We will use Proposition \ref{presingularint} to replace $I(n,q)/q$ by its asymptotic evaluation. To justify this step for a general weight function $w$, we first explain how to reduce to weights satisfying the hypotheses of Proposition \ref{presingularint}. Fix a small $\varepsilon>0$ and let $w_0$ be a smooth bump function such that $w(x)=1 ,x\in [-1,1]$ and $\text{supp}(w)\subseteq [-2,2]$. Then observe that
\begin{equation*}
\tilde{S(X)}=\sum_{\substack{\textbf{x}\in\mathbb{Z}^4\\F(\textbf{x})=0}} A(x_1)w\left(\frac{\textbf{x}}{X}\right)= \sum_{\substack { \textbf{x}\in\mathbb{Z}^4\\F(\textbf{x})=0}}A(x_1)w\left(\frac{\textbf{x}}{X}\right)w_0\left(\frac{F(\textbf{x})}{X^2\varepsilon}\right).
\end{equation*}
Since the new weight function $w(\textbf{x})w_0(F(\textbf{x})/\varepsilon)$ is $\varepsilon^{-1}$-inert, all the error term bounds are changed by at most an extra power of $\varepsilon^{-1}$. Since $\varepsilon$ will be chosen small and independently of $X$, the error term bounds remain essentially unchanged. On the other hand we claim that for $\textbf{t}\in \text{supp}(w(n/X,\textbf{t})w_0(F(n/X,\textbf{t})/\varepsilon)) $ we have $ |\nabla G|\gg_{\varepsilon} 1, $ where $G(\textbf{t})=F(n/X,\textbf{t})$. Note that the implied constant above is uniform in $n\sim X$. To this end, firstly note that since the form $F$ is given by non-singular symmetric matrix $M$, 
\begin{equation}\label{gradf}
            |\nabla F(\textbf{x})|=|2M\textbf{x}|\gg 1
\end{equation}
 for $\textbf{x}$ bounded away from $\textbf{0}$. Now using Euler's identity $\sum_{i}x_i\frac{\partial F}{\partial x_i}=2F(\textbf{x})$ at the special points $(n/X,\textbf{t})$ gives
\begin{equation*}
 \sum_{i} t_i\frac{\partial G}{\partial t_i}+(n/X)\left. \frac{\partial F}{dx_1}\right|_{\textbf{x}=(n/X,\textbf{t})}\ll \varepsilon
\end{equation*}
for $\textbf{t}\in \text{supp}(w(n/X,\textbf{t})w_0(F(n/X,\textbf{t})/\varepsilon)) $. Also, from \eqref{gradf} we know
 \begin{equation}\notag
 |\nabla G|+\Big| \frac{\partial F}{dx_1}|_{\textbf{x}=(n/X,\textbf{t})}\Big|\gg 1.
 \end{equation} The claim follow after combining the last two inequalities.  With the above justification, we can substitute Proposition \ref{presingularint} in \eqref{tailseparated} to get
        \begin{equation*}
            M(X)=X\sum_{n\geq 1}\sum_{q\leq Q^{1-\varepsilon}}A(n)\sigma_{\infty}(n/X)S_q(n)/q^3+O(X^{5/2}Q^{-2-1/8+\varepsilon}\|A\|).
        \end{equation*}Here we have ignored the error part coming from Proposition \ref{presingularint} since it is negligibly small in the region $q\leq Q^{1-\varepsilon}$. Using Proposition \ref{eq:truncated-singular-series} for the $q$-sum and arguing as earlier for the error part we thus get
 \begin{equation}\label{finalmainterm}
 M(X)=X\sum_{n\geq 1}A(n)\sigma_{\infty (n/X)} \mathfrak{S}_F(n) +O(X^{5/2}Q^{-2-1/8+\varepsilon}\|A\|).
 \end{equation}This gives the main term in Theorem \ref{mainresult}. As we shall see, the coefficients $\mathfrak{S}_F(n)$ admit a convenient factorization into prime-power components and from Proposition \ref{presingularint} we know $\sigma_{\infty}(x)$ is smooth, the $n$-sum above can be evaluated explicitly in terms of $X$ whenever the coefficients $A(n)$ possess suitable arithmetic structure. In Section \ref{refinement}, we will obtain a full asymptotic formula when the coefficients $A(n)$ are multiplicative and the associated $L$-function has sufficiently good analytic properties.


\section{Initial reduction of \texorpdfstring{$E(X)$}{E(X)}}Recall that
\begin{align*}
E(X) = \frac{1}{Q}\sum_{q\leq Q}\frac{1}{q^4}\sum_{x_1 \in \Z} A(x_1)V(x_1/X)\sum_{\substack{\textbf{c}' \in \Z^3\\ \textbf{c}' \neq \textbf{0}}} S_q(x_1,\textbf{c}') I_q(x_1,\textbf{c}'). 
\end{align*}Since the coefficients $A(x_1)$ are essentially arbitrary, our main focus is to exploit cancellation in the sum over $q$ in $\sum_q S_q(x_1,\mathbf{c}')$ for fixed $x_1$ and $\mathbf{c}'\neq\mathbf{0}$. Although the character sum $S_q$ has a complicated form, its Fourier coefficients as a function of $x_1\bmod q$ behave nicely: they are essentially Ramanujan sums. We use this decomposition to analyze cancellation in the sum over $q$. Using periodicity we write
\begin{align*}
    S_q(x_1, \textbf{c}')= \frac{1}{q} \sum_{c_1 \modq}  \mathfrak{S}_q(c_1,\textbf{c}') e\left(  \frac{-x_1 c_1}{q}\right),  
\end{align*}
where 
    \begin{equation}\label{characterforerror}
        \mathfrak{S}_q(c_1,\textbf{c}')=  \mathfrak{S}_q(\textbf{c})= \sum_{b_1 \modq} S_q(b_1, \textbf{c}') e\left(  \frac{c_1 b_1}{q}\right) =  \sideset{}{^*}\sum_{a\modq}\sum_{\textbf{b} \modq }e\l(\frac{aF(\textbf{b})+\textbf{b}\cdot\textbf{c}}{q}\r).
    \end{equation}The point in expanding as above is the evaluation (see \cite[Lemma 26]{HB})
\ba \label{charactersumestimation}
\mathfrak{S}_q(\mathbf{c}) = q^{2} \l(\frac{\det(M)}{q} \r) c_q(F^*(\mathbf{c})), \,\,\,(q, 2 \det(M)) = 1.
\ea These character sums also satisfy multiplicativity 
\begin{equation}\label{multi}
    \mathfrak{S}_{q_1q_2}(\mathbf{c})=\mathfrak{S}_{q_1}(\mathbf{c})\mathfrak{S}_{q_2}(\mathbf{c}),\,\,\,(q_1,q_2)=1
\end{equation}and the uniform upper bound
\ba \label{generalc}
\mathfrak{S}_q(\mathbf{c}) \ll_{\det(M)} q^{3},
\ea which is provided the in same reference above. 
With this we write
    \begin{align} \label{mainerror}
    E(X) = \frac{1}{Q}\sum_{q\leq Q}\frac{1}{q^5} \sum_{\substack{\textbf{0} \neq \textbf{c}' \in \Z^3 }} \sum_{c_1 \modq} \sum_{x_1 \in \Z} A(x_1)V(x_1/X)  e\left(  \frac{-x_1 c_1}{q}\right)   \mathfrak{S}_q(\textbf{c})  I_q(x_1,\textbf{c}'). 
\end{align}We first eliminate the contribution from small moduli $q$ by means of a trivial estimate. To this end, we decompose
\ba \label{Errordecompose}
E(X) = E_{\text{small}}(X) + E_{\textnormal{large}}(X),
\ea where $E_{\mathrm{small}}(X)$ and $E_{\mathrm{large}}(X)$ denote the contributions to $E(X)$ from the ranges $q\le Q_{\mathrm{small}}$ and $q>Q_{\mathrm{small}}$, respectively. Here $Q_{\mathrm{small}}>0$ is a parameter to be chosen later.

\section{Analysis of The integral \texorpdfstring{$I_q$}{Iq} }
In this section, we obtain a more refined analysis of the integral $I_q(x_1,\textbf{c}')$ which will be needed later to bound the terms for small values of $q$. Recall that we have
\ba \label{Iq}
I_q(x_1,\textbf{c}') = \int_{\R^3} \int_{\R} g(q,y) w\l(\frac{(x_1, \textbf{t})}{X}\r) e\l(\frac{F(x_1, \textbf{t})y}{qQ} - \frac{\textbf{c}' \cdot \textbf{t}}{q}\r) \diff y\,\, \diff\textbf{t}, 
\ea 

By repeated integration by parts, it is easy to see that $\mathcal{I}_q(\textbf{c})$ is negligibly small unless $c_2,c_3,c_4\ll 1$.
\begin{lemma}\label{refineintegral}
Suppose that $\|\textbf{c}'\|_{\infty}= \max_{2\leq i\leq 4}\{|c_i|\}\gg 1$. Then for any $x_1\in\mathbb{Z}$, we have 
\[
I_q(x_1,\textbf{c}') \ll \frac{X^{4+\varepsilon}}{q}\left(\frac{q}{Q}\right)^{5/2+ \ve}.
\]
\end{lemma}
\begin{proof}
As with the character sum $S_q$, we work with the Fourier transform of $I_q(x_1,\textnormal{c}')$ which has nicer structure. To this end, by Fourier inversion we get
\begin{equation}\label{fourierinv}
    I_q(x_1,\textbf{c}')= \frac{X^4}{q}\int_{c_1}\mathcal{I}_q(\textbf{c}) e\l(\frac{x_1c_1}{q}\r)\,dc_1,
\end{equation}where
\begin{equation}\label{I_q}
\mathcal{I}_q(\textbf{c})=\int_{\mathbb{R}}\int_{\mathbb{R}^4}g(q,y)w\l(\textbf{x}\r) e\l(\frac{yX^2F(\textbf{x})}{qQ} - \frac{X(\textbf{x}.\textbf{c})}{q}\r)\, \diff y \,\,\diff \textbf{x},
\end{equation}where as always $\textbf{c}=(c_1,c_2,c_3,c_4)$ and $c_1$ is the  dual variable after Fourier inversion in $x_1$. We can assume that $q\ll Q^{1-\varepsilon}$, since otherwise the claim is trivially true using the trivial bound $|\mathcal{I}_q(\textbf{c})|\ll 1 $.  Once again, integration by part in the $c_1$ variable gives that the integral has arbitrary power savings unless $c_1\ll 1$. In this region, we apply the weighted stationary phase analysis using Lemma \ref{thm:exponentialintegral} iteratively to the variables $x_1, x_2, x_3, x_4$. The complete separation of variables in the phase function, resulting from the diagonalization $M_F = R^{\mathsf{T}}DR$, allows us to process the inner integrals independently.

Without loss of generality, let $c_2=\max_{1\leq i\leq 4}\{|c_i|\}\gg 1$. Then, by comparing the size of the terms in the phase function of the $x_1$-integral in \eqref{I_q}, we conclude that $\mathcal{I}_q(\textbf{c})$ is negligibly small unless
\begin{equation}
\frac{yX^2}{qQ}\gg \frac{Xc_2}{q}\gg Q^{\varepsilon}\Rightarrow y\gg 1.
\end{equation}

With the above data, we perform some elementary transformations on the integral $\mathcal{I}_q(\textbf{c})$. Suppose the form $F$ is given by the matrix
\begin{equation}\label{formmatrix}
    M_F=R^{\mathsf{T}}DR
\end{equation}
where $R$ and $D$ are rational matrices such that $D$ is diagonal and $\det(R)=1$. Hence 
\begin{equation*}
    F(\textbf{x})=(R\textbf{x})^{\mathsf{T}}D (R\textbf{x}).
\end{equation*}We perform a change of variable $\textbf{x}\mapsto R^{-1}\cdot(X\textbf{x})$ and introduce a dyadic partition in each variable so that
\begin{equation}\label{y_i}
\begin{aligned}
         \mathcal{I}_q(\textbf{c})&\ll \sup_{\substack{r_i\ll 1\\r\sim 1 }}\left|\int_{y\sim r}\int_{x_i\sim r_i}g(q,y)w(R^{-1}\textbf{x})e\l(\frac{X^2y(\textbf{x}^{\mathsf{T}}D\textbf{x})}{qQ} - \frac{X((R^{-1})^T\textbf{c})\cdot\textbf{x}}{q}\r)\diff y\,\diff \textbf{x}\right|\\
         &=\sup_{\substack{r_i\ll 1\\r\sim 1}}\left|\int_{y\sim r}\int_{x_i\sim r_i}g(q,y)w(R^{-1}\textbf{x})e\l(\sum_{i=1}^{4}f_i(x_i,y)\r)\diff y\,\diff \textbf{x}\right|
\end{aligned}
\end{equation}where 
\begin{equation*}
f_i(x_i,y) = \frac{\beta_i y X^2x_i^2}{qQ} - \frac{X z_ix_i}{q},
\end{equation*}
where
\begin{equation}\label{lineartrans}
\text{Diag}(\beta_1,\beta_2,\beta_3,\beta_4)=D\,\,\,\text{and}\,\,\,(z_1,z_2,z_3,z_4)= (R^{-1})^T \textbf{c}.
\end{equation}We apply the stationary phase analysis using Lemma \ref{thm:exponentialintegral} iteratively to the variables $x_1, x_2, x_3, x_4$. The complete separation of variables in the phase function, resulting from the diagonalization $M_F = R^{\mathsf{T}}DR$, allows us to process the inner integrals independently. To this end, for a 1-inert smooth function $H(x_i,x_{i+1},\cdots,x_4,y)$ supported on $x_j\sim r_j,y\sim r$, consider the oscillatory integral
\begin{equation}\notag 
\int_{\mathbb{R}}H(x_i,x_{i+1},\cdots,x_4,y)e\l( f_i(x_i,y) \r)\diff x_{i}
\end{equation}Note that the phase function satisfies the derivative bound
\begin{equation}\notag
\frac{\partial^{a_i+\cdots+a_4+a_5}}{\partial x_i^{a_i}\cdots\partial x_4^{a_4}\partial y^{a_5}}f_i(x_i,y)\ll \frac{yX^2}{qQ}\frac{1}{r_i^{a_i}\cdots r_4^{a_4}r^{a_5}},\,\,r_j,r\ll 1
\end{equation}and that
\begin{equation}\notag
\frac{\partial^2}{\partial x_i^2}f_i(x_i,y)\asymp \frac{yX^2}{qQ}.
\end{equation}Hence applying Lemma \ref{thm:exponentialintegral} with the unique stationary point $t_0=\frac{z_i Q}{2 \beta_i y X}$ and $Y=\frac{yX^2}{qQ}\gg Q^{\varepsilon}$ we obtain
\begin{equation}\label{iterative}
\int_{\mathbb{R}}H(x_i,x_{i+1},\cdots,x_4,y)e\l( f_i(x_i,y) \r)\diff x_{i}=\sqrt{\frac{qQ}{yX^2}}G(x_{i+1},\cdots,x_4,y)e\left(\frac{z_i^2Q}{4\beta_iyq}\right)+O_A(Q^{-A})
\end{equation}for some 1-inert function $G$ supported on $x_j\sim r_j, y\sim r$. Applying the formula \eqref{iterative} iteratively to our integral \eqref{y_i} starting with the $x_1$-integral we obtain
\begin{equation}\label{formulaaplied}
\begin{aligned}
&\int_{y\sim r}\int_{x_i\sim r_i}g(q,y)w(R^{-1}\textbf{x})e\l(\sum_{i=1}^{4}f_i(x_i,y)\r)\diff y\,\diff \textbf{x}\\
&=\left(\frac{qQ}{X^2}\right)^{2}\int_{y\sim r}y^{-2}g(q,y)G(y)e\left(\frac{Q}{4yq}\sum_{i=1}^4\frac{z_i^2}{\beta_i}\right)+O_A(Q^{-A}),
\end{aligned}
\end{equation}where $G$ is 1-inert and supported on $y\sim r$. Changing variable $y\mapsto ry$ in the last $y$-integral and then doing integration by part once we get that
\begin{equation}\label{intbyparts}
\int_{y\sim r}y^{-2}g(q,y)G(y)e\left(\frac{Q}{4yq}\sum_{i=1}^4\frac{z_i^2}{\beta_i}\right)\ll (q/Q)(\sum_{i=1}^4z_i^2/\beta_i)^{-1}
\end{equation}Combining \eqref{intbyparts}, \eqref{formulaaplied} and \eqref{y_i} we conclude
\begin{equation}
 \mathcal{I}_q(\textbf{c})\ll (q/Q)^{3}(\sum_{i=1}^4z_i^2/\beta_i)^{-1}.
\end{equation}Now observing that $$\sum_{i=1}^4z_i^2/\beta_i=\textbf{c}^{\mathsf{T}} M_F^{-1} \textbf{c} $$ following from the relations \eqref{lineartrans} and \eqref{formmatrix}, we thus get
\begin{equation*}
     \mathcal{I}_q(\textbf{c})\ll (q/Q)^{3}F^*(\textbf{c})^{-1}
\end{equation*}where $F^*$ is the dual form of $F$. Combining the above with the trivial bound \eqref{formulaaplied}, we arrive at
\begin{equation}\label{measure}
\begin{aligned}
  |I_q(x_1,\textbf{c}')|&\ll \frac{X^{4+\varepsilon}}{q}\l(\frac{q}{Q}\r)^2\int_{c_1\ll 1}\min\l\{\frac{q}{Q}|F^*(\textbf{c})|^{-1},1\r\}\,dc_1\\
  &\ll \frac{X^{4+\varepsilon}}{q}\l(\frac{q}{Q}\r)^2\sup_{R\ll 1}\operatorname{meas}\{c_1: F^*(c)\sim R\}\min\l\{\frac{q}{QR},1\r\}.
  \end{aligned}
\end{equation}Note that $F^*(\textbf{c})$ is a quadratic polynomial in $c_1$ with bounded coefficients, whose leading coefficient is non-zero by our assumption on $F$ (as $ \det(N) \neq 0$). Hence, in the worst case (namely, when $F^*(\textbf{c})$ has a double root), the condition $F^*(\textbf{c})\sim R$ implies that $x=x_0+O(R^{1/2})$, where $x_0$ is the double root of $F^*(x,\textbf{c}')$. This translates to $c_1=x_0+O(R^{1/2})$, and therefore $\operatorname{meas}\{c_1: F^*(c)\sim R\}\ll R^{1/2}$. Hence
\begin{equation*}
    \sup_{R\ll 1}\operatorname{meas}\{c_1: F^*(c)\sim R\}\min\l\{\frac{q}{QR},1\r\}\ll \sup_{R\ll 1}\min\l\{\frac{q}{QR^{1/2}},R^{1/2}\r\}\ll \l(\frac{q}{Q}\r)^{1/2}.
\end{equation*}The lemma follows after substituting the above bound in \eqref{measure}.
\end{proof}
Without any restriction, we have the following bound for our integral $\mathcal{I}_q(\mathbf{c}) $.
\begin{lemma}\label{upperboundintegral} 
We have 
\ba
I_q(x_1,\textbf{c}') \ll \frac{X^{4+\varepsilon}}{q}\left(\frac{q}{Q}\right)^{2+ \ve} .
\ea 
\end{lemma}
\begin{proof}
It follows trivially after estimating the step \eqref{formulaaplied}.
\end{proof}

\section{Estimating $E_{\text{small}}(X)$}
We proceed for the estimation of the contribution from small $q$. We have
\begin{align} 
    E_{\text{small}}(X) = \frac{1}{Q}\sum_{q\leq Q_{\text{small}}}\frac{1}{q^5} \sum_{\substack{\textbf{c}' \in \Z^3\\ \textbf{c}' \neq \textbf{0}}} \sum_{c_1 \modq} \sum_{x_1 \in \Z} A(x_1)V(x_1/X)  e\left(  \frac{-x_1 c_1}{q}\right)   \mathfrak{S}_q(\textbf{c})  I_q(x_1,\textbf{c}'). \notag
\end{align}

Taking absolute value everywhere in above equation  and using Lemma \ref{refineintegral} for the integral and \ref{charactersumestimation} for the character sum we get that
\begin{equation}\label{smallq}
    \begin{aligned}
        E_{\text{small}}(X)&\ll \frac{X^4}{Q^{7/2}}\sum_{\textbf{0} \neq \textbf{c}'\ll 1}\sum_{x_1}|A(x_1)V(x_1/X)|\sum_{q\leq Q_{\text{small}}}\,\,\sum_{c_1\pmod q}\frac{(q, F^*(\textbf{c}))}{q^{3/2}}.
    \end{aligned}
\end{equation}We divide the $q$-sum into dyadic blocks $q\sim C\ll Q_{\text{small}}$, extend the final sum over $c_1\pmod q$ above to $c_1\sim C$ by positivity, and interchange the $q$- and $c_1$-sums. Since the coefficient of $c_1$ in $F^*(\mathbf{c})$ is nonzero, there are at most $O(1)$ values of $c_1\sim C$ for which $F^*(\mathbf{c})=0$, with the remaining coordinates fixed. Hence
\begin{equation*}
    \sum_{q\leq Q_{\text{small}}}\,\,\sum_{c_1\modq}\frac{(q, F^*(\textbf{c}))}{q^{3/2}}\ll \sup_{C\ll Q_{\text{small}}}C^{-3/2}\l(\sum_{q\sim C} 1+\sum_{\substack{c_1\sim C\\ F^*{(\textbf{c})\neq 0}}}\sum_{q\sim C}(q, F^*(\textbf{c})\r)\ll Q_{\text{small}}^{1/2}.
\end{equation*}Substituting the last bound into \eqref{smallq} we conclude
\begin{align} \label{Esmallfinalbd}
 E_{\text{small}}(X)\ll X\l(\frac{Q_{\text{small}}}{Q}\r)^{1/2} \sum_{n\in\mathbb{Z}} \left|A(n) V\l( \frac{n}{X}\r)\right|.
\end{align}

\section{Analysis of $E_{\text{large}}(X)$}
We proceed for the estimation of
\begin{equation}\label{analysisofElarge}
E_{\text{large}}(X) = \frac{1}{Q}\sum_{Q_{\text{small}}\leq q\leq Q}\frac{1}{q^5} \sum_{\substack{\textbf{0} \neq \textbf{c}' \in \Z^3}} \sum_{c_1 \modq} \sum_{n \in \Z} A(n)V(n/X)  e\left(  \frac{-n c_1}{q}\right)   \mathfrak{S}_q(\textbf{c})  I_q(n,\textbf{c}'). 
\end{equation}The overall strategy is to exploit cancellation in the sum over $q$ after using the explicit evaluation of the character sum $\mathfrak{S}_q$ from Lemma \ref{charactersumestimation}. To this end, we first separate the bad primes and write $q=q_1q_2$ with $q_1|(2\text{det}(M_F))^{\infty}$, $(q_2,(2\text{det}(M_F)))=1$. Using the multiplicative property from \eqref{multi} and the evaluation from \eqref{charactersumestimation} we know $\mathfrak{S}_{q_1q_2}(\textbf{c})=\mathfrak{S}_{q_1}(\textbf{c})\mathfrak{S}_{q_2}(\textbf{c})=\mathfrak{S}_{q_1}(\textbf{c})q_2^2\chi_M(q_2)c_{q_2}(F^*(\textbf{c}))$. Opening up the Ramanujan sums $c_{q_2}(F^*(\textbf{c}))=\sum_{d|(q_2,F^*(\textbf{c}))}d\mu(q_2/d)$, writing $q_2=d\lambda$, rearranging the summations we obtain
\begin{equation}\label{first E large}
\begin{aligned}
 E_{\text{large}}(X)&=\frac{1}{Q}\sum_{\substack{ 0 \neq |\textbf{c}'|\ll 1}}\,\,\sum_{\substack{q_1\leq Q\\q_1|(2\text{det}(M_F))^{\infty}}}\frac{1}{q_1^{5}}\sum_{n\in \Z}A(n) V\l( \frac{n}{X}\r)\mathop{\sum_{Q_{\text{small}}<d\lambda q_1\leq Q }}_{ \substack{ (d \lambda, 2\det(M_F))= 1 \\  }} \\ &\,\,\sum_{\substack{c_1\pmod {d\lambda q_1}\\F^*(\textbf{c}) \equiv 0 \pmod d}}\mathfrak{S}_{q_1}(\textbf{c})\frac{d\mu(\lambda)}{(d\lambda)^{3}}\chi_{M}(d\lambda)e\l(\frac{- c_1 n}{d\lambda q_1 }\r)\,I_{d\lambda q_1}(n,\textbf{c}') ,
 \end{aligned}
\end{equation}

The remaining analysis proceeds in two cases. For larger values $q_1> Q^{\delta}$ or $\lambda>Q^{\delta}$, and constituently $D$ smaller, the entire contribution can be absorbed into the error term. For smaller values of $q_1,\lambda\ll Q^{\delta}$ ($D$ generic), we do a more elaborate analysis of our sum depending on the reducibility of $F^*(x)$. Hence, we recast $E_{\text{large}}(X,s)$ as 
\ba \label{decomposition}
E_{\text{large}}(X) = E_{\delta}(X) +  E_{\text{irred.}}(X)+E_{\text{red.}}(X),
\ea 
where $E_{\delta}(X)$ denotes the contribution from $\max\{\lambda, q_1\} > Q^{\delta}$, $E_{\text{irred.}}(X)$ denotes the contribution from the range $\max\{\lambda, q_1\} \leq Q^{\delta}$ in which $F^*(x)$ is an irreducible polynomial in $\mathbb{Z}$, and $E_{\text{red.}}(X)$ denotes the contribution from the range $\max\{\lambda, q_1\} \leq Q^{\delta}$ in which $F^*(x)$ is reducible over $\mathbb{Z}$.

We end this section with an estimation for the non-generic part $E_{\delta}(X)$.
\subsection{Analysis of $E_{\delta}(X)$}
\begin{proposition}\label{analysisofEdelta}
We have 
\ba \label{Edeltafinalbd}
E_{\delta}(X) \ll  X^{1-\delta/2+\varepsilon}\|A\|.
\ea
\end{proposition}
\begin{proof}
Let us define 
\begin{equation}\label{setofroots}
    S_d = \{ \alpha \pmod d \mid F^*(\alpha,\textbf{c}') \equiv 0 \pmod d \}.
\end{equation}
We split $c_1\pmod {d\lambda q_1}$ sum in \eqref{first E large} into residue classes $c_1=dk+\alpha, \alpha\in S_d$ and $1\leq k\leq \lambda q_1$ so that
\begin{align}\label{reciprocity}
e\left(\frac{-c_1n}{d\lambda q_1}\right)  =  e\left(\frac{- k n}{\lambda q_1}\right) e\left(\frac{-\alpha n}{d\lambda q_1}\right)
\end{align}Now since $(\lambda,q_1)=1$, the sum $k\pmod {\lambda q}$ can be split and the sum over $k\pmod \lambda$ forces $\lambda|n$. Using this observation and the bound $\mathfrak{S}_{q_1}(\cdots)\ll q_1^{5/2}$ from \eqref{prop:bad-prime-bound} and taking absolute values in \eqref{first E large} we get
\begin{align*}
    E_{\delta}(X)&\ll\frac{1}{Q}\sum_{\substack{ 0 \neq |\textbf{c}'|\ll 1}}\,\,\sum_{\substack{q_1\leq Q\\q_1|(2\text{det}(M_F))^{\infty}}}\frac{1}{q_1^{5}}\sum_{n\in \Z}|A(n) V(n/X)| \\
    & \hspace{3cm} \times \mathop{\sum_{Q_{\text{small}}<d\lambda q_1\leq Q }}_{ \substack{ \max\{\lambda, q_1\}>Q^{\delta} \\  }}\frac{q_1|S_d|}{(d\lambda)^{2}} |I_{d\lambda q_1}(n,\textbf{c}')|\delta_{\lambda|n} ,
\end{align*}
Using $I_{d\lambda q_1}(\cdots)\ll \frac{X^4}{d\lambda q_1}(d\lambda q_1/Q)^2$ form \eqref{upperboundintegral} and the fact that $|S_d|\ll 1$, the above is bounded by 
\begin{align*}
 E_{\delta}(X) &\ll  X\sum_{n}|A(n)V(n/X)|\sum_{\substack{ q_1 \leq Q \\ q_1 | 2\det(M_F)^{\infty}}} \sum_{\substack{\lambda|n\\\max\{\lambda, q_1\}>Q^{\delta}}} \frac{1}{q_1^{1/2}\lambda}\ll X^{1-\delta/2+\varepsilon}\|A\|.
\end{align*}
\end{proof}

\section{Analysis of $E_{\textnormal{irred.}}(X)$}
 We consider the sum 

\begin{align} \label{Eirred}
    E_{\text{irred.}}(X)&=\frac{1}{Q}\sum_{\substack{ 0 \neq  |\textbf{c}'|\ll 1}}\,\,\sum_{\substack{q_1\leq Q\\q_1|(2\text{det}(M_F))^{\infty}}}\frac{1}{q_1^{5}}\sum_{n\in \Z}A(n) V\l( \frac{n}{X}\r)\mathop{\sum_{Q_{\text{small}}<d\lambda q_1\leq Q }}_{ \substack{ (d \lambda, 2\det(M_F))= 1 \\ \max\{\lambda,q_1\}\leq Q^{\delta}  }} \\ 
 &\, \hspace{2cm} \times \sum_{\substack{c_1\mod(d\lambda q_1)\\F^*(\textbf{c}) \equiv 0 \mod d\\F^*(x)\,\text{irred.}}}\mathfrak{S}_{q_1}(\textbf{c})\frac{d\mu(\lambda)}{(d\lambda)^{3}}\chi_{M}(d\lambda)e\l(\frac{- c_1 n}{d\lambda q_1 }\r)\,I_{d\lambda q_1}(n,\textbf{c}') .
\end{align}

We will prove the following bound for  $ E_{\text{irred.}}(X)$. 
\begin{proposition}\label{irreducibleupperbd}
   Let the notations be as above.  We have
  \begin{equation}\label{irreduciblefinalbd}
E_{\text{irred.}}(X)\ll   X^{1+3\delta/2}(Q/Q_{\text{small}})\|A\|_2.
  \end{equation}
\end{proposition}

As earlier, we divide the $c_1$-sum into residue classes $c_1=dk+\alpha, \alpha\in S_d$ and $k\pmod {\lambda q_1}$, where $S_d$ as in \eqref{setofroots} to get
\begin{equation}\label{bilinearsum}
    \begin{aligned}
 E_{\text{irred.}}(X)&=\frac{1}{Q}\sum_{\substack{0 \neq   |\textbf{c}'|\ll 1}}\,\,\sum_{\substack{q_1\leq Q^{\delta}\\q_1|(2\text{det}(M_F))^{\infty}}}\frac{1}{q_1^{5}}\mathop{\sum}_{ \substack{  \\ \lambda\leq Q^{\delta}  }}\frac{\mu(\lambda)\chi_{M}(\lambda)}{\lambda^{3}}\sum_{\beta\pmod {\lambda q_1}} \\ 
 &\sum_{\substack{Q_{\text{small}}/\lambda q_1\leq d\leq Q/\lambda q_1\\(d,2\,\text{det}(M_F))=1}}\,\,\,\sum_{\alpha\in S_d}\,\,\sum_{n=\beta\pmod {\lambda q_1}}A(n)V(n/X)g_{\beta}(\alpha,d)e\l(\frac{- \alpha n}{d\lambda q_1 }\r)\,I_{d\lambda q_1}(n,\textbf{c}') ,
 \end{aligned}
\end{equation}where
\begin{equation}\notag
g_{\beta}(d,\alpha)= \frac{\chi_M(d)}{d^{2}}\sum_{k \pmod{\lambda q_1}}  \mathfrak{S}_{q_1}(dk + \alpha, \mathbf{c}') e\left(-\frac{k\beta}{\lambda q_1}\right).
\end{equation}Note that from Lemma \ref{prop:bad-prime-bound}
\begin{equation}\label{gbeta}
 g_{\beta}(d,\alpha)\ll   d^{-2} \lambda q_1 q_1^{5/2}=\lambda q_1^{7/2}d^{-2}.
\end{equation}Observe that, since $\lambda q_1$ can be taken as small as desired, in the absence of the factor $g_{\beta}(\alpha,d)$ we could directly obtain a power saving in the sum
\begin{equation*}
    \sum_d\sum_{\alpha\in S_d}e\left(\frac{\alpha n}{d\lambda q_1}\right)
\end{equation*}
using, for example, the results of Hooley \cite{Hoo}. The presence of $g_{\beta}$, however, requires a power saving in sums of the form
\begin{equation*}
    \sum_d\sum_{\alpha\in S_d}e\left(\frac{\overline{k}\alpha n}{d}\right),
\end{equation*}
where $k$ is fixed and small compared with $d$. Such a saving could still be obtained by extending the work of Hooley and Duke-Friedlander–Iwaniec \cite{DFIr} to account for the level aspect. To avoid these technicalities, we instead use a large sieve-type estimate for the bilinear sum over $n$ and $d$ above. This allows us to proceed without exploiting any arithmetic properties of $g_{\beta}$. To this end, we obtain the following estimate in the bilinear sum over $n$ and $d$ in the second line of \eqref{bilinearsum} using the quadratic-root large sieve due to Balog--Blomer--Dartyge--Tenenbaum.
\begin{proposition}
    Continuing with above notations , we have
\begin{align} \label{estimate:bilinearsum}
   B:=  &\sum_{\substack{Q_{\text{small}}/\lambda q_1\leq d\leq Q/\lambda q_1\\(d,2\,\text{det}(M_F))=1}}\,\,\sum_{\alpha\in S_d}\,\,\sum_{n=\beta\pmod{\lambda q_1}} A(n)V\left(\frac{n}{X}\right)I_{d\lambda q_1}(n,\textbf{c}')g_{\beta}(d,\alpha)e\left(\frac{-\alpha n}{d\lambda q_1}\right) \notag \\
& \hspace{4cm}  \ll (Q/Q_{\text{small}})  X^{2+\varepsilon} q_1^{9/2} \lambda^2 \|A\|_2.
\end{align}
    
\end{proposition}
\begin{proof}We first separate the variables $n$ and $d$ from $I_{d\lambda q_1}(n,\textbf{c}')$ using Mellin inversion. It is clear from the phase in \eqref{Iq} that  $I_{d\lambda q_1}(n,\textbf{c}')$ is $(Q/q)\sim (Q/d\lambda q_1)\ll Q/Q_{\text{small}}$-inert in $n$. Hence Mellin inversion gives  that 
\begin{equation}\label{suminq}
\begin{aligned}
  B  \ll\sum_{Q_{\text{small}}/\lambda q_1\leq d\leq Q/\lambda q_1}\,\,\sum_{\alpha\in S_d}\frac{Q}{d\lambda q_1} \sup_{\substack{\Re(s)=\varepsilon\\ |\Im (s)|\ll (Q/d\lambda q_1)}}|g_{\beta}(\alpha, d)\tilde{I}_{d\lambda q_1}(s,\textbf{c}')|\,\,|\Omega| \qquad \textrm{where}
    \end{aligned}
\end{equation}

\begin{align*}
    \Omega&=\sum_{n=\beta\pmod{\lambda q_1}} A(n)V\left(\frac{n}{X}\right)n^{-s}e\left(\frac{-\alpha n}{d\lambda q_1}\right), \qquad \textrm{and} \\
    &  \tilde{I}_{d\lambda q_1}(s,\textbf{c}')= \int_{x\sim 1}|I_{d\lambda q_1}(Xx,\textbf{c}')|(Xx)^{s-1}\,dx.
\end{align*}

Estimating the Mellin transform above trivially using Lemma \ref{refineintegral} and plugging in the bound \eqref{gbeta} we get that \eqref{suminq} is dominated by
\begin{equation}\label{fedupoflabelling}
    X^{5/2}\lambda^{3/2}q_1^{4}\sup_{\substack{\Re(s)=\varepsilon\\ |\Im (s)|\ll (Q/Q_{\text{small}})}}\sum_{Q_{\text{small}}/\lambda q_1\leq d\leq Q/\lambda q_1}d^{-3/2}\sum_{\alpha\in S_d}\,\,|\Omega|.
\end{equation}Writing $n=k \lambda q_1+\beta, k\ll X/\lambda q_1$ and applying Cauchy-Schwarz inequality we obtain
    \begin{equation}\label{csomega}
       \sum_{Q_{\text{small}}/\lambda q_1\leq d\leq Q/\lambda q_1}d^{-3/2}\sum_{\alpha\in S_d}\,\,|\Omega|\ll \frac{\lambda q_1}{Q_{\text{small}}}\l(\sum_{d\leq Q/\lambda q_1}\sum_{\alpha\in S_d}\Bigg|\sum_{k\ll X/\lambda q_1}B(k)e\left(\frac{\alpha k}{d}\right)\Bigg|^2\r)^{1/2},
    \end{equation}where $B(k)=A(\lambda q_1k+\beta)(\lambda q_1k+\beta)^{-s}V((\lambda q_1k+\beta)/X)$. Using Lemma \ref{quadractic-rootlargeseive}, we thus get that \eqref{csomega} is bounded by
    \begin{equation*}
       \ll \frac{\lambda q_1}{Q_{\text{small}}}(X/\lambda q_1)^{1/2}\|A\|_2.
    \end{equation*}The proposition follows after substituting the last bound into \eqref{fedupoflabelling}.
\end{proof}
\begin{proof}[Proof of Proposition \ref{irreducibleupperbd}]
Substituting the estimate \eqref{estimate:bilinearsum} into \eqref{bilinearsum} and trivially estimating everything else we arrive at
\begin{align*}
E_{\text{irred.}}(X) 
&\ll  \frac{1}{Q}(Q/Q_{\text{small}})\sum_{\substack{ q_1 \leq Q^{\delta} \\ q_1 | 2\det(M_F)^{\infty}}} \frac{1}{q_1^{5}}  \mathop{\sum_{\lambda\leq Q^{\delta}}} \frac{1}{\lambda^{3}}\sum_{\beta\pmod {\lambda q_1}}X^2q_1^{9/2}\lambda^2\|A\|_2  \\
&\ll X(Q/Q_{\text{small}})\|A\|_2\sum_{\substack{ q_1 \leq Q^{\delta} \\ q_1 | 2\det(M_F)^{\infty}}}q_1^{1/2}\sum_{\lambda\leq Q^{\delta}}1 \ll X^{1+3\delta/2}(Q/Q_{\text{small}})\|A\|_2.
\end{align*}
\end{proof}

\section{Analysis of $E_{\textnormal{red.}}(X)$} 

We turn to the analysis of our final bit 
\begin{equation} \label{finalbit}
   \begin{aligned}
 E_{\text{red.}}(X)&=\frac{1}{Q}\sum_{\substack{0 \neq  |\textbf{c}'|\ll 1}}\,\,\sum_{\substack{q_1\leq Q\\q_1|(2\text{det}(M_F))^{\infty}}}\frac{1}{q_1^{5}}\sum_{n\in \Z}A(n) V\l( \frac{n}{X}\r)\mathop{\sum_{Q_{\text{small}}<d\lambda q_1\leq Q }}_{ \substack{ (d \lambda, 2\det(M_F))= 1 \\ \max\{\lambda,q_1\}\leq Q^{\delta}  }} \\ 
 &\,\,\sum_{\substack{c_1\pmod {d\lambda q_1}\\F^*(\textbf{c}) \equiv 0 \pmod d\\F^*(x)\,\text{red.}}}\mathfrak{S}_{q_1}(\textbf{c})\frac{d\mu(\lambda)}{(d\lambda)^{3}}\chi_{M}(d\lambda)e\l(\frac{- c_1 n}{d\lambda q_1 }\r)\,I_{d\lambda q_1}(n,\textbf{c}') ,
 \end{aligned}
\end{equation}

We will carry out the analysis for a fixed $\textbf{c}'$ such that $F^*(x,\textbf{c}')$ is a reducible polynomial in $x$ and assume $  F^*(x,\textbf{c}') =(\alpha_1 x-\beta_1)(\alpha_2x-\beta_2),$
where $\alpha_i,\beta_i\in \mathbb{Z}$ which depends on $\textbf{c}'$. Note that by our standing assumption $\det N\neq 0$, at most one of $\alpha_1$ and $\alpha_2$ can be zero. We break down the analysis into the following cases and sub-cases.

\vspace{10pt}

\noindent 
\textbf{\textit{Case 1 :}} $\alpha_1\alpha_2(\alpha_1\beta_2-\alpha_2\beta_1)\neq 0$.\\ 
Denote $\Delta=\lambda \alpha_1\alpha_2(\alpha_1\beta_2-\alpha_2\beta_1) $. We write 
\begin{equation}\label{dres}
    d=d'h,\,\, h|\Delta^{\infty},\,\,(d',\Delta)=1.
\end{equation}The reason for extracting the factor $h$ is made clear in \eqref{gcdrelation}. Large values of $h$ needs to be treated separately. To this end fix a parameter $\delta_1>0$ to be chosen later. Taking the absolute value inside in \eqref{finalbit} and using the integral bound $I_{d\lambda q_1}(\cdots)\ll \frac{X^4}{d\lambda q_1}(d\lambda q_1/Q)^{5/2}$ from Lemma \ref{upperboundintegral} and $\mathfrak{S}_{q_1}\ll q_1^{5/2}$ from Proposition \ref{prop:bad-prime-bound}, we see that the part with $h > Q^{\delta_1}$ contributes
\begin{equation*}
    X^{1/2}\sum_{\substack{0 \neq  |\textbf{c}'|\ll 1}}\,\,\sum_{\substack{q_1\leq Q^{\delta}\\q_1|(2\text{det}(M_F))^{\infty}}}\frac{1}{q_1}\sum_{n\in \Z}|A(n) V\l( n/X\r)|  \mathop{\sum_{\lambda\leq Q^{\delta}}\,\,\sideset{}{^*}\sum_{d\leq Q/\lambda q_1}}\sum_{\substack{c_1\pmod {d\lambda q_1}\\F^*(\textbf{c}) \equiv 0 \pmod d\\}}\frac{1}{d^{1/2}\lambda^{3/2}},
\end{equation*}where the $d$-sum above is restricted to \eqref{dres} with $h>Q^{\delta_1}$.  The number of solutions in $c_1\pmod {\lambda q_1} $ to the quadratic congruence $F^*(\textbf{c})=0\pmod d$ is at most $\lambda q_1$. Hence the above is bounded by

\begin{align} \label{trivialbdlargeh}
    &X^{1/2}\sum_{\substack{q_1\leq Q^{\delta}\\q_1|(2\text{det}(M_F))^{\infty}}}\sum_{n\in \Z}|A(n) V\l( n/X\r)| \sum_{\substack{h|\Delta^{\infty}\\h>Q^{\delta_1}}} \mathop{\sum_{\lambda\leq Q^{\delta}}\,\,\sideset{}{^*}\sum_{d'\leq Q/\lambda q_1h}}\frac{1}{d'^{1/2}h^{1/2}\lambda^{1/2}}, \notag \\
      &  \hspace{2cm}\ll X\|A\|\sum_{\substack{h|\Delta^{\infty}\\h>Q^{\delta_1}}}\frac{1}{h}\ll X^{1-\delta_1}\|A\|.
\end{align}

In the region $h\leq Q^{\delta_1}$ we further factorize $d'= d_1d_2,$ such that the congruence $F^*(\textbf{c})=0\pmod {d'h}$ splits as 
\begin{equation}\label{congsplit}
    \begin{aligned}
        &\alpha_1c_1-\beta_1=0\pmod {d_1}\\
        &\alpha_2c_1-\beta_2=0\pmod {d_2}\\
        &F^*(\textbf{c})=0\pmod h.
    \end{aligned}
\end{equation}Observe the relation $ \alpha_2(\alpha_1c_1-\beta_1)-\alpha_1(\alpha_2c_1-\beta_2)=\alpha_1\beta_2-\alpha_2\beta_1.$ Now since $(d_1d_2,\alpha_1\beta_2-\alpha_2\beta_1)=1$, it follows that 
\begin{equation}\label{gcdrelation}
    (d_1,d_2)=(d_1,\alpha_2c_1-\beta_2)=(d_2,\alpha_1c_1-\beta_1)=1.
\end{equation}Rewriting \eqref{finalbit} with the above splitting and isolating the $d_1,d_2,n$ and $c_1$ sum (call it $\mathscr{S}$), we see that the part $h\leq Q^{\delta_1}$ can be written as
\begin{equation}\label{d2d1n}
   \begin{aligned}
       \frac{1}{Q}\sideset{}{^*}\sum_{\substack{ |\textbf{c}'|\ll 1\\\textbf{c}'\neq 0}}\,\,\sum_{\substack{q_1\leq Q^{\delta}\\q_1|(2\text{det}(M_F))^{\infty}}}\frac{1}{q_1^{5}}\sum_{h\leq Q^{\delta_1}}\frac{1}{h^{2}}\,\sum_{\substack{\lambda\leq Q^{\delta}\\(h\lambda,2\det(M_F))=1}}\frac{\chi(\lambda)\mu(\lambda)}{\lambda^3}\mathscr{S}, \qquad \textrm{where}
   \end{aligned}
\end{equation}
\begin{align}\label{reducible}
    \mathscr{S}&= \mathop{\sum\sum}_{\substack{Q_{\text{small}}/h\lambda q_1 \leq d_1d_2\leq Q/h\lambda q_1\\(d_1d_2,2\Delta\det(M_F))=1\\(d_1,d_2)=1}}\frac{\chi_M(d_1d_2)}{(d_1d_2)^{2}} \,\sum_{n\geq 1}A(n)V(n/X)I_{d_1d_2h \lambda q_1}(n,\textbf{c}') \notag \\
&\hspace{4cm} \times \sideset{}{^*}  \sum_{\substack{c_1\pmod {d_1d_2h\lambda q_1} \\}}\mathfrak{S}_{q_1}(\textbf{c})e\l(\frac{-c_1n} {d_1d_2h\lambda q_1}\r),
\end{align}

where the $c_1$-sum is restricted to the congruence \eqref{congsplit}. We will now focus our attention to $\mathscr{S}$.

Let us first give a brief sketch of the argument to follow. Suppose $d_1\leq d_2$. Since $h\lambda q_1$ is small, the congruence conditions in \eqref{congsplit}, together with reciprocity, allow us to essentially rewrite the additive twist $e\left(\frac{-c_1n}{d_1d_2h\lambda q_1}\right)$ as $e\l(\frac{\overline{d_2}(\beta_1\overline{\alpha_1}+\beta_2\overline{\alpha_2})n}{d_1h\lambda q_1}\r)$. The idea is then to exploit cancellation in the longer $d_2$-sum $$\sum_{d_2}\chi_M(d_2)e\l(\frac{\overline{d_2}(\beta_1\overline{\alpha_1}+\beta_2\overline{\alpha_2})n}{d_1h\lambda q_1}\r).$$ When $d_1$ is very small, we cannot expect substantial cancellation directly from this sum. Instead, Poisson summation essentially reduces to the zero frequency, producing a Gauss-type character sum which yields a factor $\chi_M(n)$. The cancellation is then obtained from the correlation sum $\sum_n A(n)\chi_M(n)$. When $d_1$ has non-trivial size, Poisson summation essentially transforms the $d_2$-sum into a Kloosterman sum modulo $d_1$, and we gain cancellation from the Weil bound for Kloosterman sums. We make these arguments precise below.

\vspace{10pt}

\noindent
Without loss of generality we assume
\begin{equation*}
    d_1\leq d_2
\end{equation*} for the remaining analysis. Let us denote
\begin{equation}\label{Rsize}
    R=hd_1\lambda q_1.
\end{equation}Using the congruence \eqref{congsplit} and the Chinese remainder theorem we get
\begin{equation}\label{reciprocityapp}
  \sum_{\substack{c_1\pmod {d_2R} \\}}\mathfrak{S}_{q_1}(\textbf{c})e\l(\frac{-c_1n} {d_2R}\r)=\sideset{}{^*}\sum_{c_1\pmod R}\mathfrak{S}_{q_1}(\textbf{c})e\left(\frac{-\overline{d_2}(\alpha_2c_1-\beta_2)n}{R\alpha_2}\right)e\left(\frac{\beta_2 n}{d_2R\alpha_2}\right),
\end{equation}where $c_1\pmod {R}$ satisfies the relation modulo $d_1$ and $h$ in \eqref{congsplit}. Substituting \eqref{reciprocityapp} into \eqref{reducible}, we get
\begin{equation}\label{Adistribution}
\begin{aligned}
\mathscr{S}=& \sum_{n\geq 1}A(n)V(n/X)\mathop{\sum\sum}_{\substack{Q_{\text{small}}/h\lambda q_1 \leq d_1d_2\leq Q/h\lambda q_1\\(d_1d_2,2\Delta\det(M_F))=1\\(d_1,d_2)=1}}\frac{\chi_M(d_1d_2)g(d_2,n)}{d_1^{2}}\\
&\hspace{4cm}\sideset{}{^*}\sum_{c_1\pmod {R}}\mathfrak{S}_{q_1}(\textbf{c}) \,e\left(\frac{-\overline{d_2}(\alpha_2c_1-\beta_2)n}{R\alpha_2}\right),
\end{aligned}
\end{equation}where $g(y,n)=y^{-2}e\l(\frac{\beta_2n}{yR\alpha_2}\r)I_{yR}(n,\textbf{c}')$. Note that from the expression \eqref{I_q} and the fact that the size of the phase $\frac{\beta_2X}{d_2R\alpha_2}\sim X/q\ll Q/Q_{\text{small}}$ it follows
\begin{equation}\label{gderbd}
    y^{j_1}x^{j_2}g^{(j_1,j_2)}(y,x)\ll_{j_1,j_2} R^{-1}X^4y^{-3}(Q/Q_{\text{small}})^{j_1+j_2},
\end{equation}where $g^{(j_1,j_2)}$ denotes the $j_1$-th and $j_2$-th partial derivatives with respect to $y$ and $x$, respectively. Lets isolate the $d_2$-sum and call it $\mathcal{D}$ given by
\begin{equation*}
  \mathcal{D}=\sum_{\substack{Q_{\text{small}}/R\leq d_2\leq Q/R\\(d_2,\,2d_1\Delta \det (M_F))=1}}\chi_M(d_2)g(d_2,n)e\left(\frac{-\overline{d_2}(\alpha_2c_1-\beta_2)n}{R\alpha_2}\right)
\end{equation*}Let $ L=[h\lambda q_1\alpha_2,2\Delta \det(M_F)].$
Then, we have $ [R\alpha_2,2d_1\Delta \det(M_F)]=Ld_1,$ with $ (d_1,L)=1.$ Note the the $d_2$-sum above is periodic with period $d_1L$. Dividing into dyadic blocks $d_2\sim D$ where $Q_{\text{small}}/R\leq D\leq Q/R$ and applying Poisson summation modulo $d_1L$, we obtain
\begin{equation}\label{d_2sumpoisson}
\mathcal D=\sum_{D\,\text{dyadic}}\frac{D}{d_1L}\sum_{\tilde d\in\mathbb Z}\mathcal{S}(\tilde{d})
\mathcal I(\tilde d), \ \ \ \ \ \textrm{where}
\end{equation}
\begin{equation}\label{Itilde}
\mathcal I(\tilde d)
=
\int_{z\sim 1}
g\left(Dz,n\right)\,
e\left(-\frac{D\tilde d z}{d_1L}\right)\,dz
\end{equation}and
\begin{equation}\label{poissonchar}
   \mathcal{S}(\tilde{d})=\sum_{l\pmod{d_1L}}\chi_M(l)\,
e\left(\frac{-\overline l(\alpha_2c_1-\beta_2)n}{R\alpha_2}+ \frac{l\tilde d}{Ld_1}
\right).
\end{equation}Since $g$ is $Q/Q_{\mathrm{small}}$-inert (see \eqref{gderbd}), repeated integration by parts shows that $\mathcal I(\tilde d)$ is negligibly small unless
\begin{equation}\label{dtilderange}
|\tilde d|
\ll (Q/Q_{\text{small}})\frac{d_1L}{D}+1.
\end{equation}

\textbf{\textit{Subcase 1 :}}\label{subcase1} $d_1<D$; $d_1 \leq Q^{\alpha}. $
Note that $$(Q/Q_{\text{small}})\frac{d_1L}{D}\ll (Q/Q_{\text{small}})\frac{(hd_1\lambda q_1)^2}{Q_{\text{small}}}\ll (Q/Q_{\text{small}})^2X^{4\delta+2\delta_1+2\alpha}X^{-1}.$$ We will choose our parameters such that
\begin{equation}\label{condition}
    (Q/Q_{\text{small}})^2<X^{1-4\delta-2\delta_1-2\alpha-\varepsilon}.
\end{equation}Hence only zero frequency survives and we get (up to negligible error) and we get
\begin{equation}\label{d_2zerofreq}
    \mathcal D =\sum_{D\,\text{dyadic}}\frac{D}{d_1L}  \mathcal I(0)\,
\mathcal S(0).
\end{equation}We apply Lemma \ref{gauss-sum} with $q\leftrightarrow\det M, R\leftrightarrow R\alpha_2,\lambda= d_1$ and $\alpha\leftrightarrow (\alpha_2c_1-\beta_2)n$. It follows that $\mathcal{S}(0)$ vanishes unless $R\alpha_2|\det (M)(\alpha_2c_1-\beta_2)n$ in which case
\begin{equation*}
   \mathcal S(0)=\frac{d_1L}{\det M}\overline{\chi}_M\l(\frac{\det M (\alpha_2\gamma-\beta_2)n}{R\alpha_2}\r)\tau(\chi_M).
\end{equation*}Writing $n=cn', c=R\alpha_2/(R\alpha_2,\det M (\alpha_2c_1-\beta_2))$, and substituting \eqref{d_2zerofreq} for the $d_2$-sum in \eqref{Adistribution} we get
\begin{equation}\label{Szerofreq}
\begin{aligned}
     \mathscr{S}&=\sum_{\substack{d_1\leq Q^{\alpha}\\(d_1,\,2\Delta\det (M_F))=1}}\frac{\chi_M(d_1)}{d_1^{2}}\sideset{}{^*}\sum_{c_1\pmod R}\mathfrak{S}_{q_1}(\textbf{c})\\
     &\sum_{\substack{D\,\text{dyadic}\\Q_{\text{small}}/R\leq D\leq Q/R}}\frac{D\tau(\chi_M)}{\det M}\overline{\chi}_M\l(\frac{\det M (\alpha_2c_1-\beta_2)c}{R\alpha_2}\r)
    \sum_{n}A(cn')\chi_M(n')W(cn'/X),
     \end{aligned}
\end{equation}where $W(x)=V(x)\int_{z\sim 1}g(Dz, Xx)dz$. Using the derivative bounds for $g$ from \eqref{gderbd}), integrating by parts we obtain
\begin{align*}
    \sum_{n'}A(cn')\chi_M(n')W(cn'/X)&=\int_{t\ll X/c}\l(\sum_{n\leq t}A(cn)\chi_M(n)\r)\frac{\partial (W(ct/X))}{\partial t}\,dt \\
    &\ll X^4(Q/Q_{\text{small}})\frac{1}{RD^3}E_A(X,R), \qquad \textrm{where}
\end{align*}

\begin{equation*}
    E_A(X,R)=\sup_{c\ll R,\,\,t\ll X/c}\left|\sum_{n\leq t}A(cn)\chi_M(n)\right|.
\end{equation*}Substituting the above into \eqref{Szerofreq} we get
\begin{align*}
     \mathscr{S}&\ll X^4\sum_{\substack{d_1\leq Q^{\alpha}\\}}\frac{1}{d_1^2}\,\,\sideset{}{^*}\sum_{c_1\pmod R}|\mathfrak{S}_{q_1}(\textbf{c})| D\frac{1}{RD^3}(Q/Q_{\text{small}})E_A(X,hd_1\lambda q_1)\\
    & \hspace{3cm} \ll \left(\frac{Q}{Q_{\text{small}}}\right)^3 X^{2+\alpha}q_1^{5/2}(h\lambda q_1)^2 E_A(X,X^{2\delta+\delta_1+\alpha}).
\end{align*}
where we used the upper bounds for the character sum $|\mathfrak{S}_{q_1}|\ll q_1^{5/2}, $ $ h\leq Q^{\delta_1}, $ $ \lambda q_1\leq Q^{2\delta},  $  and $ D^{-2}\leq (Q/Q_{\text{small}})^2X^{-2+4\delta}$. Substituting the last upper bound into \eqref{d2d1n}, we conclude that the contribution of this sub-case towards \eqref{finalbit} is at most
\begin{equation}\label{d_1smallfinalbd}
    \begin{aligned}
         & X^{1+\alpha} \l(\frac{Q}{Q_{\text{small}}}\r)^3 E_A(X,X^{2\delta+\delta_1+\alpha})\,\,\sum_{\substack{q_1\leq Q^{\delta}\\q_1|(2\text{det}(M_F))^{\infty}}}\frac{1}{q_1^{5}}\sum_{h\leq Q^{\delta_1}}\frac{1}{h^2}\sum_{\lambda\leq Q^{\delta}}\frac{1}{\lambda^3}q_1^{5/2}(h\lambda q_1)^2\\
        &\ll X^{1+\delta_1+\alpha} \l(\frac{Q}{Q_{\text{small}}}\r)^3 E_A(X,X^{2\delta+\delta_1+\alpha}).
    \end{aligned}
\end{equation}

\vspace{10pt}
\textbf{\textit{Subcase 2: }} $Q^{\alpha}<d_1\leq D$.
Using $(d_1,L)=1$ and splitting the residue classes modulo $d_1L$ in \eqref{poissonchar} by the Chinese remainder theorem, \eqref{d_2sumpoisson} becomes
\begin{equation}\label{DafterCRT}
\mathcal D =\sum_{\substack{D\,\text{dyadic}\\Q_{\text{small}}/R\leq D\leq Q/R}}\frac{D}{d_1L} \sum_{|\tilde d|\ll (Q/Q_{\text{small}})d_1L/D+1} \mathcal I(\tilde d)\,
\mathcal S(\tilde d), \ \ \ \ \ \textrm{where} 
\end{equation}

\begin{equation}\label{Stildedef}
\begin{aligned}
\mathcal S(\tilde d)=S\left(-\overline{h\lambda q_1\alpha_2}(\alpha_2c_1-\beta_2)n,
\overline L\tilde d;d_1\right) \sum_{l\pmod L} \chi_M(l)\, e\left( \frac{-\overline{d_1l}(\alpha_2c_1-\beta_2)n} {h\lambda q_1\alpha_2}+
\frac{\overline{d_1}l\tilde d}{L} \right).
\end{aligned}
\end{equation}Since $(\alpha_2c_1-\beta_2,d_1)=1$, Weil's bound for Kloosterman sums gives
\begin{equation*}
    \mathcal S(\tilde d)\ll L \sqrt{d_1}\sqrt{(n,d_1)}.
\end{equation*}Using $\mathcal{I}(\tilde{d})\ll \int_{z\sim 1}|g(Dz,n)|dz\ll X^4R^{-1}D^{-3}$ from \eqref{gderbd} and substituting the last two bounds into \eqref{DafterCRT} we then get
\begin{equation*}
\begin{aligned}
    \mathcal{D}&\ll \sum_{\substack{D\,\text{dyadic}\\Q_{\text{small}}/R\leq D\leq Q/R}}\frac{D}{d_1L} \left(\l(\frac{Q}{Q_{\text{small}}}\r)\frac{d_1L}{D}+1\right)L d_1^{1/2}(n,d_1)^{1/2} X^4 R^{-1}D^{-3}\\
    &\ll \sum_{\substack{D\,\text{dyadic}\\Q_{\text{small}}/R\leq D\leq Q/R}} \l(\frac{Q}{Q_{\text{small}}}\r) \frac{D}{d_1L}L^2d_1^{1/2}(n,d_1)^{1/2} X^4 R^{-1}D^{-3}\\
    &\ll \l(\frac{Q}{Q_{\text{small}}}\r)^3X^2(h\lambda q_1)^2d_1^{1/2}(n,d_1)^{1/2},
    \end{aligned}
\end{equation*}where we have used the upper bounds $d_1<D, L\ll h\lambda q_1$. Substituting this final bound for the $d_2$-sum in \eqref{Adistribution} and using the trivial bound $\sum^*_{c_1\pmod R}|\mathfrak{S}_{q_1}|\ll q_1^{5/2} R/d_1=q_1^{5/2} h\lambda q_1$ following from the upper bound $|\mathfrak{S}_{q_1}|\ll q^{5/2}$  the relation \eqref{congsplit} we obtain
\begin{equation*}
\begin{aligned}
    \mathscr{S}&\ll \l(\frac{Q}{Q_{\text{small}}}\r)^3X^2(h\lambda q_1)^3q_1^{5/2}\,\sum_{n}|A(n)W(n/X)|\sum_{Q^{\alpha}<d_1\leq (D/h)^{1/2}}d_1^{-3/2}(n,d_1)^{1/2}\\
    &\ll \l(\frac{Q}{Q_{\text{small}}}\r)^3X^{2-\alpha/2}\|A\| (h\lambda q_1)^3q_1^{5/2}.
    \end{aligned}
\end{equation*}Substituting the last upper bound into \eqref{d2d1n}, we conclude that the contribution of this sub-case towards \eqref{finalbit} is at most

\begin{align} \label{d_1largebd}
   & \l(\frac{Q}{Q_{\text{small}}}\r)^3X^{1-\alpha/2}\|A\|\sum_{\substack{q_1\leq Q^{\delta}\\q_1|(2\text{det}(M_F))^{\infty}}}\frac{1}{q_1^{5}}  \sum_{h\leq Q^{\delta_1}}\frac{1}{h^2}\,\sum_{\substack{\lambda\leq Q^{\delta}\\}}\frac{1}{\lambda^3}(h\lambda q_1)^3q_1^{5/2}\\
    &\ll X^{1+3\delta/2+2\delta_1-\alpha/2}(Q/Q_{\text{small}})^3\|A\|.
\end{align}

\vspace{10pt}
\noindent 
\textbf{\textit{Case 2 :}} $\alpha_1\alpha_2\neq 0$ and $\alpha_1\beta_2=\alpha_2\beta_1$.\\
       Let $\gamma_i=(\alpha_i,\beta_i), i=1,2$. The hypothesis forces $\alpha_1/\gamma_1=\alpha_2/\gamma_2$ and $\beta_1/\gamma_1=\beta_2/\gamma_2$. Consequently,
       \begin{equation*}
           F^*(x,\textbf{c}')=\gamma_1\gamma_2((\alpha_1/\gamma_1)x-(\beta_1/\gamma_1))^2.
       \end{equation*}We write $d=d'h_1h_2$, where $h_2$ is square-full, $d'h_1$ is the square-free part factorised in such a way that $h_1|(\alpha_1\gamma_1\gamma_2)^{\infty}$ and $(d',\alpha_1\gamma_1\gamma_2)=1$. Since $h_2$ is square-full, we can write $h_2 = r^2$. Suppose $h_2 > Q^{\alpha}$, where the parameter $\alpha$ is chosen same as earlier in \hyperref[subcase1]{Sub-case 1}. In this case, trivially bounding as in \eqref{trivialbdlargeh}, the contribution of such terms is 
       \begin{align*}
           &X^{1/2} \sum_{\substack{q_1\leq Q^{\delta}\\q_1|(2\text{det}(M_F))^{\infty}}}\sum_{n\in \Z}|A(n) V\l( n/X\r)| \sum_{\lambda\leq Q^{\delta}}\,\,\sum_{\substack{r>Q^{\alpha/2}\\}}  \,\,\mathop{\sum_{d'h_1\leq Q/r^2\lambda q_1}}\frac{1}{{d'}_1^{1/2}h_1^{1/2}r\lambda^{1/2}}\\
       &\ll X\|A\|\sum_{\substack{r>Q^{\alpha/2}\\}}\frac{1}{r^2}\ll X^{2-\alpha/2}\|A\|.
       \end{align*}Comparing with the bound obtained earlier, this contribution is absorbed into the error term in \eqref{d_1largebd}.
       
Now suppose $h_1 > Q^{\delta_1}$, where $\delta_1$ is the same as in \eqref{trivialbdlargeh}. Once again, an estimate similar to \eqref{trivialbdlargeh} shows that this contributes
\begin{equation*}
\ell X^{2-\delta_1}\|A\|.
\end{equation*}
This is absorbed into \eqref{trivialbdlargeh}.

Now we are left with the case where $h_1\leq Q^{\delta_1}$, $h_2\leq Q^{\alpha}$, and $d'\sim D/h_1h_2$ is large. Observe that, since $d'$ is square-free,
\[
((\alpha_1/\gamma_1)x-(\beta_1/\gamma_1))^2=0\pmod {d'}
\Longleftrightarrow
(\alpha_1/\gamma_1)x-(\beta_1/\gamma_1)=0\pmod {d'}.
\]
The analysis in this case is essentially the same as in \hyperref[subcase1]{Sub-case 1}, where the congruence condition reduces to a linear congruence modulo $d_2$, where $d_2$ was large. Consequently, the upper bound for this contribution is obtained by the same argument, with the correspondence
\[
d=d_2d_1h\ \longleftrightarrow\ d=d'h_2h_1,
\]
where $d_2$ is replaced by $d'$, $d_1$ by $h_2$, and $h$ by $h_1$. Following the same line of argument for $d'$-sum, and estimating the smaller variables $h_2$ (after taking absolute values, we may drop the square-full condition) and $h_1$ in the same manner as the variables $d_1$ and $h$ in \eqref{d_1smallfinalbd}, we obtain exactly the same bound as in \eqref{d_1smallfinalbd}.

We collect the bounds obtained above in \eqref{trivialbdlargeh}, \eqref{d_1smallfinalbd}, and \eqref{d_1largebd}, and summarize them in the following proposition.
\begin{proposition}
    We have
    \begin{equation}\label{reduciblefinalbd}
    \begin{aligned}
        &E_{\text{red.}}(X)\ll X\l(\|A\|X^{-\delta_1}+X^{\delta_1+\alpha}\l(\frac{Q}{Q_{\text{small}}}\r)^3E_A(X,X^{2\delta+\delta_1+\alpha})\r.\\
        &\hspace{7cm}\l.+X^{3\delta/2+2\delta_1-\alpha/2}\l(\frac{Q}{Q_{\text{small}}}\r)^3\|A\|\r).
        \end{aligned}
    \end{equation}
\end{proposition}

\section{Balancing the Error Terms}
With the various error-term contributions established, we now choose the auxiliary parameters and derive the final bound for $E(X)$. We start with $E_{\text{large}}(X,s)$ which has the decomposition
\begin{equation*}
    E_{\text{large}}(X) = E_{\delta}(X) +  E_{\text{irred.}}(X)+E_{\text{red.}}(X)
\end{equation*}given by \eqref{decomposition}. Let $c_Q=Q/Q_{\text{small}}$. Substituting the corresponding bounds obtained in \eqref{Edeltafinalbd}, \eqref{irreduciblefinalbd} and \eqref{reduciblefinalbd} we thus get
\begin{equation*}
\begin{aligned}
    E_{\text{large}}(X)&\ll X\|A\|\left(X^{-\delta/2}+c_QX^{3\delta/2}\|A\|_2\|A\|^{-1}\right.\\
&\hspace{3cm}\left.+X^{-\delta_1}+c_Q^{3}X^{\delta_1+\alpha}E_A(X,X^{2\delta+\delta_1+\alpha})\|A\|^{-1}+c_Q^3X^{3\delta/2+2\delta_1-\alpha/2}\right). 
    \end{aligned}
\end{equation*}We do not attempt to optimize the exponent, and instead focus on obtaining a cleaner form of the error term. Choose $\delta_1=\delta/2$ and $\alpha=6\delta$  so that the three term with powers of $X$ inside the parenthesis above are equal to get
\begin{equation}\label{Elargefinalub}
\begin{aligned}
    E_{\text{large}}(X)\ll  X\|A\|\left(c_Q^3X^{-\delta/2}+c_QX^{3\delta/2}\|A\|_2\|A\|^{-1}+c_Q^3X^{13\delta/2}E_A(X,X^{17\delta/2})\|A\|^{-1}\right).
    \end{aligned}
\end{equation}We are now in position to use the above bound in the decomposition \eqref{Errordecompose} :
\begin{equation*}
E(X) = E_{\text{small}}(X) + E_{\textnormal{large}}(X),
\end{equation*}where from \eqref{Esmallfinalbd} we have the upper bound
\begin{equation*}
    E_{\text{small}}(X)\ll  X \|A\|c_Q^{-1/2}. 
\end{equation*}It remain to choose $c_Q$. For simplicity, we  just equate the bound in the display with the first term in \eqref{Elargefinalub} and choose $c_Q=Q/Q_{\text{small}}=X^{\delta/7}$ to get
\begin{equation*}
    E(X)\ll X\|A\|\left(X^{-\delta/14}+X^{23\delta/14}\|A\|_2\|A\|^{-1}+X^{97\delta/14}E_A(X,X^{17\delta/2})\|A\|^{-1}\right).
\end{equation*}Replacing $\delta\leftrightarrow\delta/9$, the above can be dominated by
\begin{equation*}
    E(X)\ll X\|A\|\left(X^{-\delta/130}+X^{\delta/4} \|A\|_2 \|A\|^{-1}+X^{\delta}\mathcal{E}_A(X,X^{\delta})\|A\|^{-1}\right).
\end{equation*}Observe that throughout our analysis we have chosen the auxiliary variables to be powers of $X$, say $X^\alpha$, with power savings in mind. However, as is clear from the proof, nowhere do we actually require the exponent $\alpha>0$ to be fixed. Indeed, one may replace the auxiliary powers $X^\alpha$ by a general parameter $R>1$ and carry out the same analysis. This flexibility is particularly important when the discrepancy $\mathcal{E}_A(X,D)$ admits only logarithmic savings, in which case the auxiliary parameters should be chosen as powers of $\log X$. Thus, replacing $X^\delta$ by a general parameter $R>1$, we obtain the more general formula
\begin{equation}\label{erroriwthR}
    E(X)\ll X\|A\|\left(R^{-1/130}+R^{1/4}\|A\|_2\|A\|^{-1}+ RE_A(X,R)\|A\|^{-1}\right),
\end{equation}for any $R>1$.

Proof of Theorem \ref{mainresult} follows after combining \eqref{finalmainterm} and \eqref{erroriwthR}.

      \section{Further refinement of the main term}\label{refinement}

     In this section, we explicitly evaluate the main term in \eqref{finalmainterm} for several classes of sequences encompassing many important examples arising in counting problems. Let us recall that
       \begin{equation}\notag
\mathfrak{S}_{F}(n)
=
\prod_{p}\sigma_{p}(n),\quad\quad\sigma_{p}(n):=\sum_{k=0}^{\infty}
\frac{S_{p^{k}}(n)}{p^{3k}}
\end{equation}where
\begin{equation}\notag 
S_q(n)=\sideset{}{^*}\sum_{a\bmod q}\sum_{{\bf b} \bmod q}
e\!\left(\frac{aF(n,{\bf b})}{q}\right), \ \ \ \textrm{with} \ \ {\bf b} = (b_2, b_3, b_4)
\end{equation}Since $F$ is homogeneous, by changing variable $\mathbf{b}\rightarrow r\mathbf{b}$, it is clear that for $(r,q)=1$ we have $S_q(rn)=S_q(n)$ and consequently
\begin{equation}\label{factoringout}
    S_{p^k}(n)=S_{p^k}(p^{\nu_p(n)}),
\end{equation}for each prime $p$. Hence $\sigma_p(n)=\sigma_p(p^{\nu_p(n)})$ for each prime $p$ and consequently
\begin{equation}\label{factorisation}
   \mathfrak{S}_{F}(n)=\prod_{p}\sigma_p(p^{\nu_p(n)}).
\end{equation}For $p\nmid 2n\det N\det M$, the local $\sigma_p(n)$ factor is dominated by $1+\chi(p)/p$ (see Lemma \ref{prop:good-prime-power}). In order to work with absolutely converging products, we instead consider
\begin{equation*}
    \mathfrak{S}_{F}(n)=L(1,\chi)\tilde{\mathfrak{S}}_{F}(n),
\end{equation*}where 
\begin{equation}\label{twistedsigma}
    \tilde{\mathfrak{S}}_{F}(n)=\prod_{p}(1-\chi(p)/p)\sigma_p(n):=\prod_{p}\tilde{\sigma}_p(n).
\end{equation}The factorization \eqref{factorisation} then extends to
\begin{equation}\label{factor2}
    \tilde{\mathfrak{S}}_{F}(n)=\prod_{p}\tilde{\sigma}_p(p^{\nu_p(n)}).
\end{equation}Note that from Lemma \eqref{prop:bad-prime-bound} and \eqref{prop:good-prime-power} we have that for $p\nmid 2n\det M\det N$
\begin{equation}\label{localfactoreval}
   \tilde{\sigma}_p(p^{\nu_p(n)})= \tilde{\sigma}_p(1)=1-\chi^2(p)/p^2,
\end{equation}and $  \tilde{\sigma}_p(p^{\nu_p(n)})\ll 1$ for the remaining primes. This justifies the absolute convergence of the product $\prod_{p}\tilde{\sigma}_p(n)$.
\subsection{The case of multiplicative coefficients $A(n)$}

\begin{theorem} \label{multiplicative-main-term}
 Suppose $A(n)$ is multiplicative and satisfy $A(p^k)\ll_{\varepsilon} p^{k(\theta+\varepsilon)}$ for some $\theta>0$. Assume that the associated Dirichlet series
\begin{equation*}
    L(A,s)=\sum_{n\geq 1}A(n)n^{-s},\,\Re(s)>1+\theta
\end{equation*}extends meromorphically to the plane $\Re(s)>\theta>0$, with at most a simple pole at $s=\sigma_0 \in (\theta, 1+\theta]$. Assume also that $L(A,s)$ has at most polynomial growth in $|\Im (s|)$ on every closed vertical strip contained in this half-plane, away from its possible pole, and that each Euler factor $A_p(s)=\sum_{k\ge0}A(p^k)p^{-ks}$ is non-vanishing for $\Re(s)\geq\theta$. Then, for a 1-inert function $W$ compactly supported on $\mathbb{R}_{\geq 0}$ and $X>1$, we have

\begin{equation*}
 \sum_{n\geq 1}A(n)\mathfrak{S}_{F}(n)W(n/X)=X^{\sigma_0}\mathfrak{C}_{F,A}+ O_{\varepsilon}(X^{\theta+\varepsilon}),
\end{equation*}where 
\begin{equation*}
   \mathfrak{C}_{F,A}= L(1,\chi)\tilde W(\sigma_0)H(A,s_0)\, {Res}_{s=\sigma_0} L(A,s)
\end{equation*}and
\begin{equation*}
    H(A,s):=\prod_{p}\frac{\sum_{k\geq   0} A(p^k) \tilde{\sigma}_p(p^{k})p^{-ks}}{A_p(s)}, 
\end{equation*}which converges absolutely for $\Re(s)>\theta$ and $\tilde{W}(s)=$ denotes the Mellin transform of W.
\end{theorem}
\begin{proof}
The evaluation \eqref{localfactoreval} and the factorization \eqref{factor2} allow us to write
\begin{equation*}
 \sum_{n\geq 1}\frac{A(n) \mathfrak{S}_{F}(n)}{n^s}=L(1,\chi) \sum_{n\geq 1}\frac{A(n) \tilde{\mathfrak{S}}_{F}(n)}{n^s}=L(1,\chi)\prod_{p}\left(\sum_{k\geq 0}A(p^k)\tilde{\sigma}_p(p^{k})p^{-ks}\right),
\end{equation*}for all $\Re(s) >1+\theta$. Since each Euler factor $A_p(s)$ is non-vanishing in $\Re(s) >\theta$,
\begin{equation*}
     \sum_{n\geq 1}\frac{A(n) \tilde{\mathfrak{S}}_{F}(n)}{n^s}=L(A,s)H(A,s),
\end{equation*}where $L(A,s)$ is the usual Dirichlet series $\sum_{n\geq 1}A(n)n^{-s}$ associated to $A$ and
\begin{equation}\label{correctingfactor}
    H(A,s):=\prod_{p}\frac{\sum_{k\geq 0}A(p^k)\tilde{\sigma}_p(p^{k})p^{-ks}}{A_p(s)}.
\end{equation}We next show that the product in $H(A,s)$ is absolutely convergent in $\Re(s)>\theta$ and hence the analytic property of $ \sum_{n\geq 1}\frac{A(n) \tilde{\mathfrak{S}}_{F}(n)}{n^s}$ are inherited from those of $L(A,s)$. From Propositions \eqref{prop:good-prime-power} and \eqref{prop:bad-prime-bound}, we get for $p\nmid 2\det M\det N$
\begin{equation}\label{sigma_psize}
   \tilde{\sigma}_p(1)=1-\chi^2(p)/p^2,\,\,\,\,  \textrm{and} \qquad \tilde{\sigma}_p(p^{k})=1+O(1/p), k\geq 1.
\end{equation}Hence from the individual coefficient bounds from the hypothesis it follows that
\begin{equation*}
    \sum_{k\geq 0}A(p^k)\tilde{\sigma}_p(p^{k})p^{-ks}=A_p(s)+O_{\varepsilon}(p^{-1}\sum_{k\geq 1 }p^{-k(\Re(s)-\theta-\varepsilon)})=A_p(s)+O_{\varepsilon}(p^{-(1+\Re(s)-\theta-\varepsilon)}).
\end{equation*}Also, for large enough primes $p$, $A_p(s)=1+O_{\varepsilon}(p^{-(\Re(s)-\theta-\varepsilon)})\gg 1$. Consequently, for all but finitely many primes, the factor
\begin{equation*}
   \frac{\sum_{k\geq 0}A(p^k)\tilde{\sigma}_p(p^{k})p^{-ks}}{A_p(s)} =1+O_{\varepsilon}(p^{-(1+\Re(s)-\theta-\varepsilon)}).
\end{equation*}The claim of absolute convergence of $H(A,s)$ in $\Re(s)>\theta$ follows. Note that this also proves 
\begin{equation}\label{absbd}
    H(A,s)\ll_{\theta} 1,  \qquad \textrm{for} \qquad \Re(s)> \theta.
\end{equation}
Using Mellin inversion $W(x)=\int_{\Re(s)=3/2}\tilde{W}(s)x^{s-1}\,dx$, we thus get
\begin{equation*}
   \sum_{n\geq 1}A(n)\mathfrak{S}_{F}(n)W(n/X)=L(1,\chi)\int_{\Re(s)=3/2}\tilde{W}(s) X^sL(A,s)H(A,s)\,ds.
\end{equation*}Shifting the contour to $\Re(s)=\theta+\varepsilon$ and collecting the residue at $s=\sigma_0$ we obtain, up to a negligible error,
\begin{equation}\label{residue}
\begin{aligned}
    \sum_{n\geq 1}A(n)\mathfrak{S}_{F}(n)W(n/X)=&X^{\sigma_0}L(1,\chi)\tilde W(\sigma_0)H(A,s_0)\,\operatorname*{Res}_{s=\sigma_0}L(A,s)\\
    &+\int_{\Re(s)=\theta+\varepsilon}X^{\theta+\varepsilon}|\tilde{W}(s)L(A,s)H(A,s)|\,ds.
    \end{aligned}
\end{equation}
The lemma follows from the rapid decay of $\tilde{W}(s)$ and the polynomial growth of the function $L(A,s)H(A,s)$ in $|\Im(s)|$ following from the hypothesis and \eqref{absbd}.

\begin{remark}
    If the pole has order $m>1$, then the residue in \eqref{residue} is $X^{\sigma_0}$ times a polynomial in $\log X$ of degree at most $m-1$ with coefficients in terms of the higher derivatives of $\tilde{W}(s)H(A,s)$ at $s=\sigma_0$.
\end{remark}
\end{proof}

\subsection{The case of indicator function $A(n)=\mathbbm{1}_{X\leq n\leq X+H}$}. This case can be incorporated into the multiplicative setting above by taking $A(n)=1$ and introducing a suitable cut-off in the weight $W(n/X)$. However, our aim is to obtain an asymptotic that remains effective even when $H$ is essentially as small as $1$. Introducing such cut-off functions creates additional oscillation, which in turn leads to the analysis of $\zeta(s)$ at non-trivial heights $t$. Crude bounds for $\zeta(s)$ are not sufficient to recover the optimal result in this range, and a more delicate $L$-function analysis would be required. We therefore adopt a direct approach in this case, which yields the optimal asymptotic by working explicitly with the coefficients $\mathfrak{S}_F(n)$. The idea is that the dependency of these coefficients on $n$ is roughly a divisor type function $\prod_{p\mid n}(1+1/p)\approx \sum_{d|n}h(d)$ with small coefficients $h(d)$.
\begin{theorem}\label{constantfunction}
    We have
    \begin{equation*}
        \sum_{n\leq X}\mathfrak{S}_F(n)=X\mathfrak{C}_F+O_{\varepsilon}(X^{\varepsilon}),
    \end{equation*}where
    \begin{equation*}
       \mathfrak{C}_F=L(1,\chi)\prod_{p}(1-p^{-1})\l(\sum_{k\geq 0}\tilde{\sigma}_p(p^k)p^{-k}\r).
    \end{equation*}
\end{theorem}

\begin{proof}
   Once again we consider
   \begin{equation*}
    \tilde{\mathfrak{S}}_{F}(n)=\prod_{p}\tilde{\sigma}_p(n),
\end{equation*}where $\tilde{\sigma}_p(n)$ as in \eqref{twistedsigma}, so that
\begin{equation*}
    \sum_{n\leq X}\mathfrak{S}_F(n)=L(1,\chi) \sum_{n\leq X}\tilde{\mathfrak{S}}_{F}(n).
\end{equation*}The first step is to divide separate the bad primes $p|\Delta$ where $\Delta=2\det M\det N$ and write $n=\ell n'$, $\ell|\Delta^{\infty}, (n',\Delta)=1$. Then from the factorisation \eqref{factor2} we have
\begin{equation}\label{goodbadsplit}
    \tilde{\mathfrak{S}}_{F}(\ell n')=\prod_{p\nmid n'}\tilde{\sigma}_p(\ell)\prod_{p|n'}\tilde{\sigma}_p(p^{\nu_p(n')})=\prod_{p}\tilde{\sigma}_p(\ell)\prod_{p|n'}\frac{\tilde{\sigma}_p(p^{\nu_p(n')})}{\tilde{\sigma}_p(1)},
\end{equation}where we used the fact that $\tilde{\sigma}_p(\ell)=\tilde{\sigma}_p(1)\neq 0$ since $p|n'$ and $(n',\Delta)=1$. 

Define $\{h(n)\}_{n\geq 1}$ by $h(1)=1$ and for each prime $p$ and $k>1$, $h(p^k)=\tilde{\sigma}_p(p^k)/\tilde{\sigma}_p(1)-1$. Extend $h$ to all positive integers multiplicatively. Observe that \eqref{sigma_psize} imply
\begin{equation}\label{hbd}
    |h(d)|\ll \frac{1}{\text{rad}(d)},
\end{equation}where $\text{rad}(d)=\prod_{p|d}p$. Then \eqref{goodbadsplit} translates to
\begin{equation*}
     \tilde{\mathfrak{S}}_{F}(\ell n')=\prod_{p}\tilde{\sigma}_p(\ell)\prod_{p|n'}(1+h(p^{\nu_p(n)}))=\prod_{p}\tilde{\sigma}_p(\ell)\sum_{d\|n'}h(d),
\end{equation*}where $d\|n'$ means $d|n'$ and $(n'/d,d)=1$. Consequently
\begin{equation*}
    \sum_{n\leq X}\tilde{\mathfrak{S}}_{F}(n)=\sum_{\ell|\Delta^{\infty}}\l(\prod_{p}\tilde{\sigma}_p(\ell)\r)\sum_{\substack{d\leq X/\ell\\(d,\Delta)=1}}h(d)\sum_{\substack{n'\leq X/\ell d\\(n', d\Delta)=1}}1.
\end{equation*}Using $\sum_{n\leq Y, (n,k)=1}=Y\frac{\phi(k)}{k}+O(1)$ we thus obtain

\begin{align} \label{sigmamainterm}
    \sum_{n\leq X}\tilde{\mathfrak{S}}_{F}(n)&=\frac{X\phi(\Delta)}{\Delta}\sum_{\ell|\Delta^{\infty}}\frac{\l(\prod_{p}\tilde{\sigma}_p(\ell)\r)}{\ell}\sum_{\substack{d\leq X/\ell\\(d,\Delta)=1}}\frac{h(d)\phi(d)}{d^2} \notag \\
& \hspace{30pt}+O\l(\sum_{\ell|\Delta^{\infty}}\l(\prod_{p}\tilde{\sigma}_p(\ell)\r)\sum_{d\leq X}|h(d)|\r).
\end{align}
From \eqref{hbd}, the error term above contributes at most a power of $\log X$. It remains to evaluate the main term. Note that using \eqref{hbd} once again we have $\sum_{d\geq X/\ell}|h(d)|/d\ll X^{\varepsilon}(X/\ell)^{-1} $ so that the $h$-sum in the main term in \eqref{sigmamainterm} can be extended to infinity at the cost of an error of size at most $X^{\varepsilon}$. With these observations arrive at
\begin{equation}\label{premainterm}
    \sum_{n\leq X}\tilde{\mathfrak{S}}_{F}(n)=\frac{X\phi(\Delta)}{\Delta}\sum_{\ell|\Delta^{\infty}}\frac{\l(\prod_{p}\tilde{\sigma}_p(\ell)\r)}{\ell}\sum_{\substack{d\geq 1\\(d,\Delta)=1}}\frac{h(d)\phi(d)}{d^2}+O_{\varepsilon}(X^{\varepsilon}).
\end{equation}We now show that the constant factor in the main term above exactly matches that of Theorem \eqref{multiplicative-main-term}. Using the definition of $h$, for each $p\nmid \Delta$, the Euler factor
\begin{equation*}
\begin{aligned}
    \sum_{k\geq 0}h(p^k)\phi(p^k)/p^{2k}&=1+(1-p^{-1})\sum_{k\geq 1}\l(\tilde{\sigma}_p(p^k)/\tilde{\sigma}_p(1)-1\r)p^{-k}\\
    &=\frac{(1-p^{-1})}{\tilde{\sigma}_p(1)}\l(\sum_{k\geq 0}\tilde{\sigma}_p(p^k)p^{-k}\r).
    \end{aligned}
\end{equation*}Hence
\begin{equation}\label{goodprimefactor}
    \sum_{\substack{d\geq 1\\(d,\Delta)=1}}\frac{h(d)}{d}=\prod_{p\nmid \Delta}\frac{(1-p^{-1})}{\tilde{\sigma}_p(1)}\l(\sum_{k\geq 0}\tilde{\sigma}_p(p^k)p^{-k}\r).
\end{equation}Also note that from the property \eqref{factor2}  we get $\prod_{p}\tilde{\sigma}_p(\ell)=\prod_{p\nmid \Delta}\tilde{\sigma}_p(1)\prod_{p|\Delta}\tilde{\sigma}_p(p^{\nu_p(\ell)})$, so that
\begin{equation}\label{badprimefactor}
   \frac{\phi(\Delta)}{\Delta}\sum_{\ell|\Delta^{\infty}}\frac{\l(\prod_{p}\tilde{\sigma}_p(\ell)\r)}{\ell}=\prod_{p\nmid \Delta}\tilde{\sigma}_p(1)\prod_{p|\Delta}(1-p^{-1})\l(\sum_{k\geq 0}\tilde{\sigma}_p(p^k)p^{-k}\r).
\end{equation}The theorem follows after combining \eqref{badprimefactor}, \eqref{goodprimefactor} and \eqref{premainterm}.
\end{proof}Using the above theorem, for a smooth function $g$ summation by part gives
\begin{equation*}
\begin{aligned}
    \sum_{1\leq n\leq X}\mathfrak{S}_{F}(n)g(n)&=S(X)g(X)-\int_{1}^{X}S(t)g'(t)\,dt,\\
&=\mathfrak{C}_F\int_{1}^{X}g(t)\,dt+O_{\varepsilon}\l(X^{\varepsilon}\l(\|g\|_{\infty}+\int_{1}^{X}|g'(t)|dt\r)\r)
    \end{aligned}
\end{equation*}where $S(t)=\sum_{n\leq t}\mathfrak{S}_{F}(n)$. Hence for $1\leq H\leq X$, subtracting the above asymptotic for length going up to $X+H$ and $X$ we obtain
\begin{corollary}\label{shortintervalsingularseries}
    Let $X\geq H>1$. For a 1-inert smooth function $g$ compactly supported on $\mathbb{R}_{>0}$, 
    \begin{equation*}
        \sum_{X\leq n\leq X+H}\mathfrak{S}_{F}(n)g(n)=\mathfrak{C}_F\int_{X}^{X+H}g(x)\,dx+O_{g,\varepsilon}\l(X^{\varepsilon}\r).
    \end{equation*}
\end{corollary}

\subsection{The singular series when $A=\mathbbm{1}_{m^{\ell}, m\in\mathbb{N}}$.}In this sub-section we relate the local factors $\sum_{k\geq 0}\tilde{\sigma}_p(p^{k\ell})p^{-k}$, for a fixed integer $\ell>0$, to the usual local densities. These factors appear in Theorem \ref{constantfunction} when $\ell=1$, in our applications to  $\ell$-th power solutions and soluble conic fibres. Denote
\begin{equation}\label{sigmapl}
         \sigma_{p,\ell}=\lim_{\nu\to\infty}p^{-3\nu}\#\{\textbf{b}\in (\mathbb{Z}/p^\nu\mathbb{Z})^4 : F(b_1^{\ell},b_2,b_3,b_4)\equiv 0\pmod {p^{\nu}}\},
    \end{equation}and
    \begin{equation*}
        S_{q,\ell}:=\sum_{\beta \pmod q}S(\beta^{\ell}, q)=\sideset{}{^*}\sum_{a\pmod q}\sum_{\textbf{b}\pmod q}e\l(\frac{aF(b_1^{\ell},b_2,b_3,b_4)}{q}\r).
    \end{equation*}Then a similar argument as in \eqref{padicdensity3var} gives
    \begin{equation*}
        \sum_{k\geq 0}\frac{S_{p^k,\ell}}{p^{4k}}= \sigma_{p,\ell}=\lim_{\nu\to\infty}p^{-3\nu}\#\{\textbf{b}\in (\mathbb{Z}/p^\nu\mathbb{Z})^4 : F(b_1^{\ell},b_2,b_3,b_4)\equiv 0\pmod {p^{\nu}}\}.
    \end{equation*}
\begin{proposition} \label{kthpowerseries}
    For $\ell \geq 1$. Then for each prime $p$, we have
    \begin{equation*}
        (1-p^{-1})\sum_{r\geq 0}\frac{\sigma_{p}(p^{r\ell})}{p^{r}}=\sum_{k\geq 0}\frac{S_{p^k,\ell}}{p^{4k}}= \sigma_{p,\ell}.
    \end{equation*}
\end{proposition}
\begin{proof}
    For each $\beta \pmod {p^k}$ we write $\beta=p^r\beta', (\beta',p)=1$. Grouping all such $\beta$ for a fixed $r$, using $S(p^r\beta',p^k)=S(p^r,p^k)$ we obtain
    \begin{equation*}
        S_{p^k,\ell}:=\sum_{\beta \pmod {p^k}}S(\beta^{\ell}, p^k)=S(0,p^k)+(1-p^{-1})\sum_{r=0}^{k-1}p^{k-r}S(p^{r\ell},p^k).
    \end{equation*}Note that for $r\geq k, S(p^{r\ell},p^k)=S(0,p^k)$. Hence, extending the last $r$-sum to infinity and subtracting the corresponding geometric series, we see that
    \begin{equation*}
        S_{p^k,\ell}=(1-p^{-1})p^k\sum_{r\geq 0}\frac{S(p^{r\ell},p^k)}{p^r}.
    \end{equation*}The claim follows after substituting the definition 
    \begin{equation*}
        \sigma_p(p^{r\ell})=\sum_{k\geq 0}\frac{S(p^{r\ell},p^k)}{p^{3k}}.
    \end{equation*}
\end{proof}

\subsection{The norm indicator} Fix a imaginary quadratic field $K$ with discriminant $D$ and class number one. Let $Q$ be its integer norm form. Consider
\begin{equation*}
    A(x_1)=\mathbbm{1}_{Q(u,v)=x_1}.
\end{equation*}
\begin{lemma}
   With the above notations and assumptions, we have
    \begin{equation}\label{conicmt}
        \sum_{n\geq 1}A(n)\mathfrak{S}_{F}(n)W(n/X)= \frac{X}{\sqrt{\log X}}L(1,\chi)\tilde{W}(1)C_K\,\,\sigma^*(F,K)+O\l(X\log X^{-3/2}\r),
    \end{equation}where $C_K=L(1,\chi_D)^{1/2}\prod_{\chi_D(p)=-1}(1-p^{-2})^{1/2}
\prod_{p|D}(1-p^{-1})^{-1/2}$ and $$ \sigma^*(F,K)=\prod_{p}(1-\chi(p)/p)\sigma^*_p,$$ where
\begin{equation*}
\sigma^*_p=\begin{cases}
    (1+p^{-1})\sigma_{p,2}, & \chi_D(p)=-1,\\
    \sigma_{p,1}, & \chi_D(p)\neq -1.
\end{cases}
\end{equation*}Here $\sigma_{p,2}$ and $\sigma_{p,1}$ are as in \eqref{sigmapl}. 
\end{lemma}

\begin{proof}
Since $K$ has class number one, the coefficients
\begin{equation*}
    A(x_1)=\mathbbm{1}_{Q(u,v)=x_1}
\end{equation*}are multiplicative and has the prime power evaluation
\begin{equation}\label{primepowereval}
    A(p^k)=\begin{cases}
        1, & k\,\,\text{even and}\,\,\chi_D(p)=-1\\
        0, & k\,\,\text{odd and}\,\,\chi_D(p)=-1\\
        1, & \chi_D(p)\neq -1,
    \end{cases}
\end{equation}where $D$ is the fundamental discriminant of $K$ and $\chi_D= \l(\frac{D}{\cdot}\r)$. Consequently, its associated $L$-function is
\begin{equation}\label{normfactor}
    L(A,s)=\prod_{\chi_D(p)=-1}(1-p^{-2s})^{-1}\prod_{\chi_D(p)\neq -1}(1-p^{-s})^{-1},\,\,\Re(s)>1.
\end{equation}Squaring this out and rearranging gives
\begin{equation}\label{squaredL}
    L^2(A,s)=\zeta(s)L(s,\chi_D)\prod_{\chi_D(p)=-1}(1-p^{-2s})^{-1}
\prod_{p|D}(1-p^{-s})^{-1}.
\end{equation}Hence one can write
\begin{equation}\label{sqrtLvalue}
   L(A,s) = (s-1)^{-1/2}G_Q(s),
\end{equation}where 
\begin{equation*}
    G_Q(s)=((s-1)\zeta(s)L(s,\chi_D))^{1/2}\prod_{\chi_D(p)=-1}(1-p^{-2s})^{-1/2}
\prod_{p|D}(1-p^{-s})^{-1/2}
\end{equation*}which is holomorphic near 1 (note the terms inside the square-root in $G_Q(s)$ are non-zero near 1 using the standard zero-free regions) and $$G_Q(1)=L(1,\chi_D)^{1/2}\prod_{\chi_D(p)=-1}(1-p^{-2})^{-1/2}
\prod_{p|D}(1-p^{-1})^{-1/2}.$$ Using the fact that $A(x_1)$ is multiplicative and the $L$-function given be \eqref{sqrtLvalue}, one can repeat the arguments of Theorem \ref{multiplicative-main-term} and arrive till the first line \eqref{residue} giving
    \begin{equation*}
   \sum_{n\geq 1}A(n)\mathfrak{S}_{F}(n)W(n/X)=L(1,\chi)\int_{\Re(s)=3/2}(s-1)^{1/2}\,\,\tilde{W}(s) X^s\,\,G_Q(s)H(A,s)\,ds,
\end{equation*}where $H(A,s)$ as in Theorem \ref{multiplicative-main-term}. The contour shifting argument has to be changed in this case owing to the branch point at $s=1$ instead of a pole. This time, the contour taken is around a short branch cut to the left of 1. The argument is a standard application of the Selberg–Delange method and in our case we can use Theorem II.5.2 in \cite[p.~281]{Tenenbaum2015} to get
\begin{equation*}
    \sum_{n\geq 1}A(n)\mathfrak{S}_{F}(n)W(n/X)=\frac{L(1,\chi)\tilde{W}(1)G_Q(1)H(A,1)}{\Gamma(1/2)}\frac{X}{\sqrt{\log X}}+O(X\log X^{-3/2}).
\end{equation*}The local densities appearing in the theorem are obtained by substituting the Euler factors from \eqref{normfactor} and the prime-power evaluations from \eqref{primepowereval} into the definition \eqref{correctingfactor}, and then applying the relation in Proposition \ref{kthpowerseries}.
\end{proof} 
\subsection{The case of prime counting function}
We end the examples by evaluating the summatory function of $\mathfrak{S}_{F}(n)$ over the primes. The variable $p$ below will always denote a prime. 
\begin{theorem} \label{von mongoldt theorem}
           For $X>1$, we have 
           \begin{equation*}
               \sum_{p\leq X}\mathfrak{S}_{F}(p)=\text{Li}(X)L(1,\chi)\prod_{p}\tilde{\sigma}_{p}(1)+O\l(Xe^{-\sqrt{\log X}}\r).
           \end{equation*}where $\text{Li}(x)=\int_{2}^{x}dt/\log t$.
 \end{theorem}
           \begin{proof}
               Start with 
               \begin{equation*}
    \sum_{p\leq X}\mathfrak{S}_F(p)=L(1,\chi) \sum_{p\leq X}\tilde{\mathfrak{S}}_{F}(p)
\end{equation*}First of all we ignore the bad primes $p|\Delta$ since there are at most $d(\Delta)$ of them and their contribution can be absorbed into $O(X^{\varepsilon})$. For the remaining primes $(p,\Delta)=1$, \eqref{factorisation} gives
\begin{equation*}
    \tilde{\mathfrak{S}}_{F}(p)=\l(\prod_{p'\neq p}\tilde{\sigma}_{p'}(1)\r)\tilde{\sigma}_{p}(p)=\l(\prod_{p'}\tilde{\sigma}_{p'}(1)\r)\frac{\tilde{\sigma}_{p}(p)}{\tilde{\sigma}_{p}(1)}.
\end{equation*}For $p\nmid \Delta$, Proposition \eqref{prop:bad-prime-bound} gives
\begin{equation*}
    \frac{\tilde{\sigma}_{p}(p)}{\tilde{\sigma}_{p}(1)}=\frac{\sigma_{p}(p)}{\sigma_{p}(1)}=\frac{1+1/p+(\chi(p)-1)/p^2}{1+\chi(p)/p}=1+O(1/p).
\end{equation*}Hence using the last two relation we obtain
\begin{equation*}
    \sum_{p\leq X}\tilde{\mathfrak{S}}_{F}(p)= \sum_{\substack{p\leq X\\(p,\Delta)=1}}\tilde{\mathfrak{S}}_{F}(p)+O_{\varepsilon}(X^{\varepsilon})=\l(\prod_{p'}\tilde{\sigma}_{p'}(1)\r)\pi(X)+O(\log\log X)+O(X^{\varepsilon}).
\end{equation*}The claim follows after substituting the well known approximation $\pi(X)=\text{Li}(X)+O\l(Xe^{-\sqrt{\log X}}\r)$.
\end{proof}
Let $S(t)= \sum_{p\leq t}\mathfrak{S}_{F}(p)$ and $g$ be a smooth function. Summation by part and the above estimate for $S(t)$ gives
\begin{equation*}
\begin{aligned}
    &\sum_{ p\leq X}\mathfrak{S}_{F}(p)g(p)=S(X)g(X)-\int_{1}^{X}S(t)g'(t)\,dt,\\
&=L(1,\chi)\l(\prod_{p}\tilde{\sigma}_p(1)\r)\frac{1}{\log X}\int_{1}^{X}g(t)\,dt+O\l(Xe^{-\sqrt{\log X}}\l(\|g\|_{\infty}+\int_{1}^{X}|g'(t)|dt\r)\r).
    \end{aligned}
\end{equation*}Hence we obtain
\begin{corollary}\label{primesingularseries}
   For a 1-inert smooth function $g$ compactly supported on $\mathbb{R}_{>0}$, 
    \begin{equation*}
        \sum_{p\leq X}\mathfrak{S}_{F}(p)g(p)=L(1,\chi)\l(\prod_{p}\tilde{\sigma}_p(1)\r)\frac{1}{\log X}\int_{1}^{X}g(t)\,dt+O_{g}\l(Xe^{-\sqrt{\log X}}\r).
    \end{equation*}
\end{corollary}

 \appendix

\section{Character sum in the main term}  
Let us return to the quadratic form 
\begin{align}
    F(\textbf{x})=\textbf{x}^{T}M\textbf{x},\,\,\,\textbf{x}=(x_1,x_2,x_3,x_4).
\end{align}considered in the paper. Here $M$ is a symmetric integer matrix with $\det M$ non-zero and non-square. Recall that we are also assuming that the sub-matrix obtained by deleting first row and column of $M$ is non-singular. Hence we can write 
\begin{align}\label{MN definition}
M=
\begin{pmatrix}
m&\textbf{r}^T\\
\textbf{r}&N
\end{pmatrix} 
\end{align}with $\det N\neq 0$, so that 
\begin{equation*}
\begin{aligned}
    F(x,{\bf b})={\bf b}^TN{\bf b}+2x\textbf{r}^T\mathbf{b}+mx^2.
\end{aligned}
\end{equation*}In this section we will obtain explicit evaluation for the associated character sum
\begin{equation*}
    S_q(n) = \sideset{}{^*}\sum_{a\modq}\sum_{\textbf{b} \modq} e\l(\frac{aF(n, \textbf{b})}{q}\r),
\end{equation*} and upper bounds for the general sum
\ba 
S_q(n,\textbf{c}) = \sideset{}{^*}\sum_{a\modq}\sum_{\textbf{b} \modq} e\l(\frac{aF(n, \textbf{b}) + \textbf{c} \cdot \textbf{b}}{q}\r).
\ea 

We begin by proving that $ S_q(n)$ is multiplicative as a function of $q$. 
\begin{lemma} \label{lem:S-multiplicative}
    For $(q_1,q_2)=1$, we have $ S_{q_1q_2}(n)=S_{q_1}(n)S_{q_1}(n).$
\end{lemma}
\begin{proof}
Note that by reciprocity we have
\begin{equation*}
    e\l(\frac{aF(n, \textbf{b})}{q_1q_2}\r)=e\l(\frac{\overline{q_2}aF(n, \textbf{b})}{q_1}\r)e\l(\frac{\overline{q_1}aF(n, \textbf{b})}{q_2}\r).
\end{equation*}The claim after follows after splitting $a\pmod {q_1q_2}$ and each $b_i\pmod {q_1q_2}$ into residue classes modulo $q_1$ and $q_2$ using the Chinese remainder theorem. 
\end{proof}

In the next lemma, we shall evaluate $ S_{p^r}(n)$ for good primes $p\nmid 2\det M\det N$. 
\begin{lemma}\label{prop:good-prime-power}
Suppose $p\nmid2\det(M)\det(N)$, where $M,N$ as in \eqref{MN definition}. Put $t=v_p(n)$. Then, for every \(r\geq1\),
\begin{align*}
    S_{p^r}(n) = S_{p^r}(p^t)=
\begin{cases}
\varphi(p^{2j})p^{3j}, & r=2j,\quad 1\leq j\leq t,
\\[2ex]
\displaystyle
\chi_M(p)p^{5t+2}, &r=2t+1,
\\[2ex]
0,& \text{\rm otherwise},  
\end{cases} 
\end{align*}

Consequently, we have the following explicit local evaluation :
\begin{equation*}
    \sigma_p(n)=\sigma_p(p^{t})=\sum_{r\geq 0}\frac{S_{p^r}(p^t)}{p^{3r}}=1+\frac{1}{p}+\frac{\chi(p)-1}{p^{t+1}}.
\end{equation*}

\end{lemma}

\begin{proof}
The first equality follows from our earlier observation in \eqref{factoringout}.  Since $p\nmid \det N$, the matrix \(N\) is invertible modulo $p^r$. Applying the change of variables $ \mathbf b \rightarrow \mathbf c= \mathbf b + N^{-1} \mathbf r$, we get   $ F(n,\mathbf b) =\left(\mathbf b+nN^{-1}\mathbf r \right)^T N \left( \mathbf b+nN^{-1}\mathbf r \right) +
\kappa n^2,$ where $ \kappa=m-\mathbf r^TN^{-1}\mathbf r.$ Again replacing $ \mathbf b+nN^{-1}\mathbf r \rightarrow \mathbf b$, we get 
\begin{align}\label{completingsquare}
    S_{p^r}(p^t) =\sideset{}{^\star}\sum_{a\bmod p^r} e\left( \frac{a\kappa p^{2t}}{p^r}\right)\sum_{\mathbf b\bmod p^r}
e\left(\frac{a\mathbf b^TN\mathbf b}{p^r}\right).
\end{align}Note that  $\kappa=m-\mathbf r^TN^{-1}\mathbf r$ is precisely the Schur complement of the block $N$ and consequently 
\begin{equation}\label{schuridentity}
    \det M=\kappa \det N.
\end{equation}Diagonalizing $N$, we see that the sum last $\mathbf{b}\mod {p^r}$ is the product of three quadratic gauss sums
\begin{equation*}
    \prod_{i=1}^{3}\l(\sum_{z_i\pmod {p^r}}e\l(\frac{au_iz_i^2}{p^r}\r)\r),
\end{equation*}where $u_1u_2u_3=\det N$. We use the well known formula
\begin{equation}\label{gausssumformula}
    \sum_{z_i\pmod {p^r}}e\l(\frac{au_iz_i^2}{p^r}\r)=\varepsilon_{p^r}p^{r/2}\l(\frac{au_i}{p^r}\r),
\end{equation}where $\varepsilon_m=1$ if $m=1\mod 4$ and $\varepsilon_m=i$ if $m=3\mod 4$. Hence the evaluation of each gauss sum depends on the parity or $r$.

\medskip

\noindent
\textbf{Case 1: \(r\) is even.} We write $ r=2j$. In this case each gauss sum is precisely $p^{3r/2}$ from the formula above. Substituting in \eqref{completingsquare} we thus obtain 
\begin{align*}
    S_{p^r}(p^t) = p^{3r/2}\sideset{}{^\star}\sum_{a\bmod p^r}
e\left(\frac{a\kappa p^{2t}}{p^r}\right) =
\begin{cases}
\varphi(p^{2j})p^{3j},&j\leq t,\\
0,&j>t
\end{cases}
\end{align*} as $ p \nmid \kappa$ from \eqref{schuridentity}.

\medskip

\noindent
\textbf{Case 2: \(r\) is odd.} In this case \eqref{gausssumformula} gives that each gauss is $\varepsilon_pp^{r/2}\l(\frac{au_i}{p}\r)$. Substituting in \eqref{completingsquare} we arrive at
\begin{align*}
    S_{p^r}(p^t)=\varepsilon_p^3 \left(\frac{\det(N)}p\right) p^{3r/2} \sideset{}{^\star}\sum_{a\bmod p^r} \left(\frac ap\right)
e\!\left( \frac{a\kappa p^{2t}}{p^r} \right)
\end{align*}By looking at each fixed residue class $a=p\lambda+r,\,0\leq\lambda\leq p^{r-1}$, it is clear the $a\pmod {p^r}$ sum above vanishes unless $ 2t=r-1$ in which case equals $p^{r-1}\varepsilon_p\sqrt{p}\l(\frac{\kappa}{p}\r)$. Substituting we obtain
\begin{align*}
    S_{p^r}(p^t)=\varepsilon_p^4\left(\frac{\det(N)\kappa}{p}\right)
p^{3r/2+r-1/2}=\left(\frac{\det(M)}p\right)p^{(5r-1)/2} = \chi_M(p)p^{5t+2}. 
\end{align*}

\end{proof}

\begin{lemma} \label{prop:bad-prime-bound}
We have  uniformly in \(q\geq1\), \(n\in\mathbb Z\), and \(\mathbf c\in\mathbb Z^3\),
\begin{align*}
    S_q(n,\mathbf c):=\sideset{}{^\star}\sum_{a\bmod q}
\sum_{\mathbf u\bmod q} e\!\left(\frac{aF(n,\mathbf u)+\mathbf c\cdot\mathbf u}{q} \right) \ll_{F,\varepsilon}q^{5/2+\varepsilon}.
\end{align*}

\end{lemma}

\begin{proof}
We define, for  \((a,q)=1\), 
\begin{align*}
    T_a(n,\mathbf c;q):=\sum_{\mathbf u\bmod q}
e\!\left(\frac{aF(n,\mathbf u)+\mathbf c\cdot\mathbf u}{q}\right).
\end{align*}
Writing \(\mathbf u=\mathbf v+\mathbf h\), we obtain

\begin{align*}
    |T_a(n,\mathbf c;q)|^2=&\sum_{\mathbf h\bmod q}
e\left(\frac{a(\mathbf h^TN\mathbf h+2n r^T\mathbf h)
+\mathbf c\cdot\mathbf h}{q}\right)\sum_{\mathbf v\bmod q}
e \left(\frac{2a\,\mathbf v^TN\mathbf h}{q}\right) \\
& \hspace{2cm} \leq q^3 \#\{\mathbf h\bmod q:2N\mathbf h\equiv0\pmod q\}.
\end{align*}

Since \(\det(N)\neq 0\), the  matrix \(2N\) has Smith normal form $ U(2N)V=\operatorname{diag}(d_1,d_2,d_3),$ with $ U,V\in\mathrm{GL}_3(\mathbb Z),$ and  \(d_1d_2d_3=\pm\det(2N)\).
 Hence we obtain 
\begin{align*}
    \#\{\mathbf h\bmod q:2N\mathbf h\equiv0\pmod q\}=
\prod_{j=1}^3 (d_j,q) \leq |d_1d_2d_3| =|\det(2N)|\ll_F1.
\end{align*}
As a result, we get  $ T_a(n,\mathbf c;q)\ll_F q^{3/2},$
uniformly in \(a,n,\mathbf c\), and \(q\). Hence, we obtain that 
\begin{align*}
    |S_q(n,\mathbf c)|\leq \sideset{}{^\star}\sum_{a\bmod q}
|T_a(n,\mathbf c;q)| \ll_F q^{5/2}.
\end{align*}

\end{proof}

\begin{lemma}\label{lem:Hn-bound}
Let $ H_n(s)$ be as given in equation \eqref{eq:bad-euler-product}. 
Then, for every \(\varepsilon>0\), $ H_n(\sigma+it)
\ll_{\varepsilon,F} n^\varepsilon$ uniformly for $ -1/2 +\varepsilon\leq \sigma\leq0 $. 

\end{lemma}

\begin{proof}
We split the product into good primes dividing \(n\) and the finitely many bad primes dividing \(2\det(M)\det(N)\). Let first \(p\mid n\) and  \( p \nmid 2\det(M)\det(N)\).  Let \(v=v_p(n)\). By the Proposition \ref{prop:good-prime-power},
\[
\sum_{k=0}^{\infty}
\frac{S_{p^k}(n)}{p^{k(s+3)}}
=
1+(1-p^{-1})
\sum_{j=1}^{v}p^{-j(1+2s)}
+
\chi_M(p)p^{-v-1-(2v+1)s}.
\]
Since $ -1/2 +\varepsilon\leq \sigma\leq0 $, we have $ 1+2\sigma\geq2\varepsilon,$ and hence $ \sum_{j=1}^{v}
\left|p^{-j(1+2s)}\right|
\leq
\sum_{j\geq1}p^{-2\varepsilon j}
\ll_\varepsilon1.$ Also. $ \left|
p^{-v-1-(2v+1)s}
\right|
\leq
p^{-1/2-\varepsilon}.$ Moreover, $ \left|
1+\chi_M(p)p^{-1-s}
\right|
\geq
1-p^{-1/2-\varepsilon}
\gg_\varepsilon1.$ Hence, the local factor satisfies $ H_{n,p}(s)\ll_\varepsilon1$ uniformly in \(v\) and \(t\). Hence, we get 
\begin{align*}
    \prod_{\substack{p\mid n\\p\nmid 2\det(M)\det(N)}}
|H_{n,p}(s)| \ll_\varepsilon C_\varepsilon^{\omega(n)}
\ll_\varepsilon n^\varepsilon.
\end{align*}

Now, let \(p\mid 2\det(M)\det(N)\). Using the bound $ S_{p^K}(n)\ll_F p^{5k/2},$ from Proposition \ref{prop:bad-prime-bound}, we obtain
\begin{align*}
    \sum_{k=0}^{\infty}\left|\frac{S_{p^K}(n)}{p^{k(s+3)}}\right|
\ll_F \sum_{k=0}^{\infty} p^{-k(1/2+\sigma)} \ll_{\varepsilon,F}1.
\end{align*}
The denominator is again bounded away from zero on the strip, and
there are only finitely many such primes. Hence, their total
contribution is $ O_{\varepsilon,F}(1).$ Combining the two estimates gives the result.

\end{proof}

\section{Some correlations with the quadratic character $\chi_M$ }
As our main result depends on the correlation $A(n)\chi_M(n)$, we estimate this the for the applications discussed in Section \ref{introandresult}.

\subsection{Restriction to a short interval}

Let $ A(n)=\mathbf 1_{(X, \ X+H]}(n),$ with  \(1\leq H\leq X\). Then we have  $ \|A\|\asymp H,$ and  $ \|A\|_2\asymp H^{1/2}.$ We prove the following lemma. 

\begin{lemma}\label{short-interval-correlation}
Let $ A(n)=\mathbbm {1}_{(X,X+H]}(n),$ be the indicator function of the interval $ (X,X+H]$  and let \(E_A(X,R)\) and \(\mathcal R_A(X)\) be as in Theorem~\ref{mainresult}. Then $ E_A(X,R)\ll_F 1.$ Also, choosing  
$ R=H^{80/41},$ we get
$$ \mathcal R_A(X) \ll_F XH\left(X^{-1/8}+H^{-1/41}\right).$$

\end{lemma}

\begin{proof}
The indicator function  restricts \(n\) sum to an interval. Hence  $ \sum_{n\leq t}A(cn)\chi_M(n)$ is a character sum over an interval. Since \(\chi_M\) is a  non-principal character  and  modulus of $ \chi_M$ is fixed,  we get  $ \sum_{n\leq t}A(cn)\chi_M(n)\ll_F1.$  Taking the supremum over $c$ and $t$, we obtain $ E_A(X,R)\ll_F 1.$ Also, in this case, we have $ \|A\|\asymp H,$ and  $ \|A\|_2\asymp H^{1/2}.$  Substituting these estimates into Theorem~\ref{mainresult} and choosing $ R=H^{130/131}$, we get 

\begin{align*}
  \mathcal R_A(X) & \ll_F XH\left( X^{-1/8} +R^{-1/130} +R^{1/16} H^{-1/2} +R^{1/2}H^{-1} \right) \\ 
  & \hspace{50pt} \ll_F XH\left(X^{-1/8}+H^{-1/131}\right).
\end{align*}

\end{proof}

\subsection{$ k$th power for $k$ odd}
Let $k$ be an odd integer and let $ A(n)$ be the indicator function for the $k$th power. We have following lemma
\begin{lemma}\label{kth-power-correlation}
Let \(k\geq 3\) be odd and $ A(n)=\mathbbm{1}_{\{n=m^k, m\in\mathbb{N}\} }$. Let \(E_A(X,R)\) and \(\mathcal R_A(X)\) be as given in
Theorem~\ref{mainresult}. Then $ E_A(X,R)\ll_F1.$ We also have 
\begin{align*}
\mathcal R_A(X) \ll_F X^{1+1/k}\left(X^{-1/8}+X^{-1/(131k)}\right).
\end{align*}
\end{lemma}

\begin{proof}
 Let  \(c\geq1\) has the prime factorization  $ c=\prod_p p^{\alpha_p}$.  We define  $ \kappa_k(c)$ as  define
\[
\kappa_k(c):=\prod_p p^{\beta_p}, \qquad 0\leq\beta_p<k, \qquad
\alpha_p+\beta_p\equiv0\pmod k.
\]
Then we have  $ A(cn)=1$ if and only if  $ n=\kappa_k(c)m^k$ for some \(m\geq1\). Hence we get 

\begin{align*}
   \sum_{n\leq t}A(cn)\chi_M(n) =   \chi_M(\kappa_k(c)) \sum_{m\leq (t/\kappa_k(c))^{1/k}}\chi_M(m^k) = \chi_M(\kappa_k(c)) \sum_{m\leq (t/\kappa_k(c))^{1/k}}\chi_M(m) \ll_F1
\end{align*}
as $ k$ is odd and conductor of $ \chi_M$ is fixed. Taking the supremum over $t$ and $c$, we get  $ E_A(X,R)\ll_F 1.$ Also, we have,  $ \|A\|\asymp X^{1/k}$ and  $ \|A\|_2\asymp X^{1/(2k)}.$  Choosing $R=X^{130/(131k)} $, we get 

\begin{align*}
    \mathcal R_A(X) & \ll_F X^{1+1/k} \left( X^{-1/8} +R^{-1/130} +R^{1/16}X^{-1/(2k)} +R^{1/2}X^{-1/k} \right) \\
     &  \ll_F X^{1+1/k} \left( X^{-1/8}+X^{-1/(131k)} \right). 
\end{align*}

\end{proof}

\subsection{The norm function}
\begin{lemma}\label{conicerror}
   Let $K$ imaginary quadratic field with discriminant $D$, norm form $Q$ and class number one. Suppose that $\chi \chi_D$ is non-principle. Then for $A(x_1)=\mathbbm{1}_{Q(u,v)=x_1}$, we have
    \begin{equation*}
        \mathcal{R}_A(X)\ll_B X (\log X)^{-B},
        \end{equation*} where $\mathcal{R}_A(X)$ as in Theorem \ref{mainresult}.
\end{lemma} 
\begin{proof}
For a fixed $c$, denote
\begin{equation*}
    d_c=\prod_{\substack{\nu_p(n)\,\text{odd}\\\chi_D(p)=-1}}p.
\end{equation*}Then from the evaluation \eqref{primepowereval} it follows
\begin{equation*}
    A(cn)=\begin{cases}
        A(n/d_c), & d_c|n\\
        0, & d_c\nmid n.
    \end{cases}
\end{equation*}Hence for any $t>0$
\begin{equation*}
    \sum_{n\leq t}A(cn)\chi(n)=\chi(d_c)\sum_{m\leq t/d_c}A(m)\chi(m).
\end{equation*}Now the same argument as in \eqref{squaredL} gives that the $L$-function corresponding to the twisted coefficient $A(m)\chi(m)$ is
\begin{equation*}
    (L(s,\chi)L(s,\chi\chi_D))^{1/2}\prod_{\chi_D(p)=-1}(1-p^{-2s})^{-1/2}
\prod_{p|D}(1-p^{-s})^{-1/2}
\end{equation*}which is holomorphic near 1 since both $\chi$ and $\chi\chi_D$ are non-principle by assumption. Hence Theorem II.5.2 in \cite[p.~281]{Tenenbaum2015} in this case gives
\begin{equation*}
    \sum_{m\leq t/d_c}A(m)\chi(m)\ll_B \frac{t/d_c}{(\log (t/d_c))^{B}}\ll_B \frac{t}{(\log t)^B}
\end{equation*}uniformly for all $t\gg 1$. Consequently we have
\begin{equation*}
    E_A(X,R)=\sup_{c\ll R, t\leq X/c}\l| \sum_{n\leq t}A(cn)\chi(n)\r|\ll_B \sup_{c\ll R}\frac{X/R}{(\log X/R)^B}\ll_B \frac{X}{(\log X)^B}.
\end{equation*}The claim follow using the above bound for $E_A(X,R)$ and choosing any $R=X^{\theta}, \theta>0$ in $\mathcal{R}_{A}(X)$.
    
\end{proof}

\subsection{The von Mangoldt function}

We now  consider the choice $ A(n)=\Lambda(n).$ We prove the following lemma.

\begin{lemma}\label{Lambda-correlation}
Let \(A(n)=\Lambda(n)\), and let $ E_{\Lambda}(X,R)$   and $ \mathcal R_{\Lambda}(X)$ be as given in theorem \ref{mainresult}. Then, for some \(c_F, \  c_F' >0\),
\begin{align*}
    E_{\Lambda}(X,R) \ll_F X\exp\!\left(-c_F\sqrt{\log X}\right) \quad \textrm{and} \quad \mathcal R_{\Lambda}(X) \ll_F X^2\exp\!\left(-c_F'\sqrt{\log X}\right).
\end{align*}

\end{lemma}

\begin{proof}
For \(c=1\), the prime number theorem for the character \(\chi_M\) gives
\[
\sum_{n\le t}\Lambda(n)\chi_M(n) \ll_F X \exp \left(-c_F\sqrt{\log X}\right),
\]
 as $ t\ll X$.  If \(c>1\), then \(\Lambda(cn)\neq0\) only when \(cn\) is a prime power. Hence we get  \(c=p^a\) for some prime \(p\); and  \(n=p^j\). Hence we get 
\begin{align*}
    \sum_{n\le t}\Lambda(cn)\chi_M(n)\ll(\log p)\left(1+\frac{\log t}{\log p}\right) \ll \log X.
\end{align*}
Combining the above two estimates we get  \(E_{\Lambda}(X,R) \ll \exp \left(-c_F\sqrt{\log X}\right)\). We also have  $ \|\Lambda\|\asymp X,$ and   $ \|\Lambda\|_2\ll (X\log X)^{1/2}.$ Substituting this into the expression of $  R_{\Lambda}(X)$ given in the theorem Theorem~\ref{mainresult}, and choosing $ R=\exp\left(\eta\sqrt{\log X}\right)$ for $ 0<\eta<2c_F$,   we get 
\begin{align*}
\mathcal R_{\Lambda}(X) & \ll_F X^2 \left(  X^{-1/8}+R^{-1/100} +R^{1/16} X^{-1/2} (\log X)^{1/2}  \right. \\
& \left.  \hspace{2cm}+R^{1/2}e^{-c_F\sqrt{\log X}} +R^{1/2}\log X/X \right) \ll_F X^2 \exp\left(-c_F'\sqrt{\log X}\right).
\end{align*}

\end{proof}

\bibliographystyle{abbrv}
\bibliography{reference}

\begin{thebibliography}{} 



\bibitem{VB}
 Blomer, V.: \emph{On triple correlations of divisor functions},  Bull. Lond. Math. Soc. {\bf 49} (2017), no. 1, 10-22. done 


\bibitem{bky}
 Blomer, V.,  Khan, R. and  Young, M. P.: \emph{Distribution of mass of holomorphic cusp forms}, Duke Math. J. { \bf  162(14)}: 2609-2644 (November 2013). done 


\bibitem{BBDT}
Balog, A.; Blomer, V.; Dartyge, C.; and  Tenenbaum, G.: \emph{Friable values of binary forms,} Comment. Math. Helv. { \bf 87}  (2012), no. 3, 639–667. done 

\bibitem{TD3}
 Browning, T.D.: \emph{The divisor problem for binary cubic form}, Journal de Th{\' e}orte des Nombres de Bordeaux {\bf 23}, 2011, 579-602.  done
 
 
 
%
%



\bibitem{HB} 
Heath-Brown, D.R.: \emph{A new form of circle method and  
its Application to quadratic forms}, J. Reine Angew. Math. {\bf 481}(1996), 149-206. MR 1421949 (97k:11139) done

\bibitem{Hoo}
Hooley, C.:  \emph{- On the number of divisors of a quadratic polynomial}, Acta Math.{ \bf  110} (1963), 97–114. done





\bibitem{HI}
Iwaniec, H.: \emph{Topics in Classical Automorphic Forms}, Graduate text in mathematics {\bf 17}, American Mathematical Society, Providence, RI, 1997.

\bibitem{IK}
Iwaniec, H. and Kowalski, E.: \emph{Analytic Number Theory}, American Mathematical Society Colloquium Publication {\bf 53}, American Mathematical Society, Providence, RI, 2004. done

\bibitem{KPY}
Kiral, Eren Mehmet; Petrow, Ian; Young, Matthew P.: \emph{Oscillatory integrals with uniformity in parameters}, J. Th{\' e}or. Nombres Bordeaux {\bf 31} (2019), no. 1, 145–159. done


\bibitem{VK} 
Kumaraswamy, V.: \emph{Counting integer points on quadrics with arithmetic weights}, Trans. Amer. Math. Soc. { \bf 373} (2020), 6929-6959. done




\bibitem{LIN}
Lin, Y.: \emph{Triple correlation of Fourier coefficents of cusp forms},  Ramanujan J. (2017),  

 http://dx.doi.org/10.1007/s11139-016-9874-1. done


 \bibitem{SD}
   Miller, S. D.: \emph{Cancellation in additively twisted sums on $GL(n)$}, American Journal of Mathematics, Vol. {\bf 128}, No. 3 (Jun. 2006), pp. 699-729. done


\bibitem{MERU0}

Meurman, T. : \emph{On exponential sums involving the Fourier coefficients of Maass wave forms}, J. Reine Angew. Math. {\bf 384} (1988), 192-207.  done


\end{thebibliography}

\end{document}